\documentclass[12pt]{amsart}
\usepackage{color,graphicx,array, amssymb, amscd,slashed}
\usepackage[all]{xy}
\usepackage{tikz-cd}
\usepackage{scalerel}
\newtheorem{Theorem}{Theorem}[section]
\newtheorem{Lemma}[Theorem]{Lemma}
\newtheorem{Proposition}[Theorem]{Proposition}
\newtheorem{Corollary}[Theorem]{Corollary}
\theoremstyle{definition}

\theoremstyle{remark}

\newtheorem{Remark}[Theorem]{Remark} 
\numberwithin{equation}{section}
\newcommand{\R}{\mathbb R}

\newcommand{\C}{\mathbb C}

\newcommand{\D}{\mathbb D}

\newcommand{\cg}{\mathcal{G}}

\newcommand{\fg}{\mathfrak g}
\newcommand{\sli}{\mathfrak{sl}}

\newcommand{\U}{{\rm U}_{3}}
\newcommand{\SU}{{\rm SU}_{3}}
\newcommand{\UI}{{\rm U}_{2, 1}}
\newcommand{\SUI}{{\rm SU}_{2, 1}}
\newcommand{\UIT}{\widetilde{ {\rm U}_{2, 1}}}
\newcommand{\SUIT}{\widetilde{ {\rm SU}_{2, 1}}}

\newcommand{\Utwo}{{\rm U}_2}
\newcommand{\UtwoT}{\widetilde {{\rm U}_{1, 1}}}
\newcommand{\id}{I}

\newcommand{\SL}{{\rm SL}_3 \mathbb C}

\newcommand{\SLR}{{\rm SL}_3 \mathbb R}
\renewcommand{\sl}{\mathfrak{sl}_3 \mathbb C}
\newcommand{\lsl}{\Lambda \mathfrak{sl}_3 \mathbb C_{\sigma}}
\newcommand{\LSL}{\Lambda {\rm SL}_3 \mathbb C_{\sigma}}

\newcommand{\ad}{\operatorname{Ad}}
\newcommand{\di}{\operatorname{diag}}
\newcommand{\tr}{\operatorname{Tr}}

\renewcommand{\Re}{\operatorname {Re}}
\renewcommand{\Im}{\operatorname {Im}}

\newcommand{\bpm}{\begin{pmatrix}}
\newcommand{\epm}{\end{pmatrix}}

\newcommand{\CH}{\mathbb {CH}^{2}}
\newcommand{\CHI}{\mathbb {CH}^{2}_1}
\newcommand{\CP}{\mathbb {CP}^{2}}
\newcommand{\f}{\mathfrak f}
\renewcommand{\d}{\mathrm d}
\newcommand{\SCP}{\scaleto{\mathbb {CP}^2}{4pt}}
\newcommand{\SCPt}{\scaleto{\mathbb {CP}^2}{3pt}}
\newcommand{\SCH}{\scaleto{\mathbb {CH}^2}{4pt}}
\newcommand{\SCHt}{\scaleto{\mathbb {CH}^2}{3pt}}
\newcommand{\ST}{\scaleto{\mathbb {CH}^2_1}{4pt}}
\newcommand{\STt}{\scaleto{\mathbb {CH}^2_1}{3pt}}

\newcommand{\SEA}{\scaleto{\mathbb {A}^3}{4pt}}
\newcommand{\SEAt}{\scaleto{\mathbb {A}^3}{3pt}}
\newcommand{\SIA}{\scaleto{i \mathbb {A}^3}{4pt}}
\newcommand{\SIAt}{\scaleto{i \mathbb {A}^3}{3pt}}

\renewcommand{\r}[1]{ #1_0}

\renewcommand{\l}{\lambda}

\usepackage{amsmath}	% required for `\equation*' (yatex added)
\begin{document}
\title{Survey on real forms of the complex 
 $A_2^{(2)}$-Toda equation and surface theory}
 
 \author[J. F.~Dorfmeister]{Josef F. Dorfmeister}
 \address{Fakult\"at f\"ur Mathematik, 
 TU-M\"unchen, 
 Boltzmannstr.3,
 D-85747, 
 Garching, 
 Germany}
 \email{dorfm@ma.tum.de}
 \author[W.~Freyn]{Walter Freyn}
 \address{Mathematical Institute,
University of Oxford, Andrew Wiles Building
 Radcliffe Observatory Quarter (550)
Woodstock Road
Oxford
OX2 6GG
}
 \email{walter.freyn@exeter.ox.ac.uk}

 \author[S.-P.~Kobayashi]{Shimpei Kobayashi}
 \address{Department of Mathematics, Hokkaido University, 
 Sapporo, 060-0810, Japan}
 \email{shimpei@math.sci.hokudai.ac.jp}
 \thanks{The third named author is partially supported by JSPS 
 KAKENHI Grant Number JP18K03265.}

\author{Erxiao Wang}
\address{Department of Mathematics, Hong Kong University of Science \& Technology, Clear Water Bay, Kowloon, Hong Kong} 
\email{maexwang@ust.hk}

 \date{\today}
\pagestyle{plain}
\begin{abstract}
 The classical result of describing harmonic maps from surfaces 
 into symmetric spaces of reductive Lie groups \cite{BuRa} 
 states that the Maurer-Cartan form with an additional parameter, 
 the so-called \textit{loop parameter},
 is integrable for all values of the loop parameter.
 As a matter of fact, the same result holds for $k$-symmetric 
 spaces over reductive Lie groups, \cite{Bu}.

 In this survey we will show that to each of the five 
 different types of \textit{real forms} for a loop group of $A_2^{(2)}$ there exists a surface class, for which some frame 
 is integrable for all values of the loop parameter if and only if it belongs 
 to one of the surface classes, that is, 
 minimal Lagrangian surfaces in $\CP$, minimal Lagrangian surfaces in 
 $\CH$, timelike minimal Lagrangian surfaces in $\CHI$, proper definite  
 affine spheres in $\R^3$ and proper indefinite  affine spheres in $\R^3$, respectively.
\end{abstract}

\maketitle
%%%%%%    TEXT START    %%%%%%%%
\section*{Introduction }
Following the important work of Zakharov-Shabat \cite{ZS} and
Ablowitz-Kaup-Newell-Segur \cite{AKNS} in the 1970s, 
systematic constructions of hierarchies of integrable differential
equations were developed. They were associated to a complex simple 
Lie algebra with various
reality conditions given by finite order automorphisms. Mikhailov \cite{Mi} first studied their reductions with various reality conditions given by finite order automorphisms. 
Drinfeld-Sokolov \cite{DS} constructed generalized 
KdV and mKdV hierarchies for any affine Kac-Moody Lie algebra using
this ZS-AKNS scheme. In particular, the sine-Gordon equation and
the sinh-Gordon equation are two real forms of the $-1$-flow or
Toda-type equation in the mKdV-hierarchy for the simplest affine
algebra $A_1^{(1)}$, which is a $2$-dimensional extension of the loop 
algebra \footnote{As in \cite[Chapter 7 or 8]{Kac} one obtains 
 a Kac-Moody algebra by adding $2$ dimensions $\C d \oplus \C c,$ 
 where for loop elements the Lie algebra element $d$ tells the 
 degree of lambda  and $c$ is central.} of $\sli_2 \C$. 

 It is amazing that these two equations have already appeared in classical
differential geometry for constant negative Gauss curvature surfaces 
(or pseudo-spherical surfaces) and constant mean curvature surfaces.
For example, B\"acklund \cite{Ba} constructed his
famous transformation for pseudo-spheres around $1883$, which
produced many explicit solutions of the sine-Gordon equation
$\omega_{xy}=\sin \omega$. This transformation and the higher flows
in the hierarchy can be regarded as hidden symmetries of such
submanifolds or differential equations. It has ever since become a central problem in geometry how to
find special submanifolds in higher dimension and/or codimension which admit similar
geometric transformations and have a lot of hidden symmetries, \cite{Te2}. It is
now natural to expect the answer to lie in integrable systems, as we
will illustrate it further using next the rank $2$ affine algebra 
$A_2^{(2)}$,  which is a $2$-dimensional extension
 of a loop subalgebra  of $\sli_3 \C$,  twisted by an outer automorphism $\sigma$, that is $\lsl$. 
 Here the outer automorphism $\sigma$ has order $6$ and it is defined by 
\[
 \sigma (g) (\l) = \hat \sigma (g(\epsilon^{-1}\l)), \quad  \mbox{for
 $g (\l) \in \Lambda \mathfrak {sl}_3 \C$},
\] with $\epsilon = e^{\pi i/3}$ (the natural primitive sixth root of unity) 
 and $\hat \sigma$ is the automorphism of $\sl$ given by
 \[
\hat \sigma (X) = - \ad (\di(\epsilon^2, \epsilon^4, -1)  \r{P} )\> X^T
\quad\mbox{with} \quad 
  \r{P} = \begin{pmatrix}
	   0 &1 &0  \\
	   1 &0 &0 \\
	   0 & 0 & -1
	  \end{pmatrix}.
\]
 Then a fundamental question for the affine algebra $A_2^{(2)}$ is,
 how many different real forms it has. In our case this means how many 
 different real forms of $\lsl$ there exist.
 The answer was given by \cite{B3R, BMR, HG}: 
 there are $5$ different real form involutions;
\begin{align*}
 (\bullet_{\SCP}) \quad &\tau(g)(\lambda) = - \overline{g(1/\bar \lambda)}^T,    \\
 (\bullet_{\SCH})\quad & \tau (g)(\lambda) =  - \ad (I_{2,1})  \overline
 {g(1/\bar \lambda)}^T,  \\
 (\maltese_{\ST}) \quad &
 \tau (g)(\lambda) = -\ad (\r{P})  \overline {g(\bar \lambda)}^T,  \\
 (\bullet_{\SEA})\quad& \tau(g)(\lambda) =\ad (I_{*}\r{P})  \> \overline
 {g(1/\bar \lambda)},  \\
 (\maltese_{\SIA}) \quad & \tau (g)(\lambda) = \overline
 {g(\bar \lambda)},
\end{align*}
 where $I_{2, 1} = \di(1, 1, -1)$ and $P_0$ is as just above.
 Moreover, $I_{*}$ denotes $\id$ or $I_{2, 1}$. 
 
 It was Tzitz\'eica \cite{Tz} who found a special class 
 of surfaces in Euclidean geometry, 
 which turns out to be equivalent to indefinite 
 \textit{affine spheres} in equi-affine geometry. They are related to 
 the real form involution
 $(\maltese_{\SIA})$ given by $\tau (g)(\lambda) = \overline
 {g(\bar \lambda)}$ above. More precisely, the coordinate frame 
 of an affine sphere
 with the  additional loop parameter
 is fixed  by the above real form involution. 
 More recently, minimal Lagrangian surfaces in $\CP$ or special Lagrangian cone in $\C^3$ have been 
 related to the involution $(\bullet_{\SCP})$ given by 
 $\tau(g)(\lambda) = - \overline{g(1/\bar \lambda)}^T$, see \cite{McIn} or 
 \cite{DM}. 

 In this survey, we relate all real forms of the affine algebra $A_2^{(2)}$ to classes of surfaces:

\begin{tabular}{lcl}
 $(\bullet_{\SCP})$& $- \overline{g(1/\bar \lambda)}^T, $&  
  Minimal Lagrangian surfaces in $\CP,$ \cite{MM}$,$ \\
 $(\bullet_{\SCH})$ & $ - \ad (I_{2,1})  \overline
 {g(1/\bar \lambda)}^T,$ &
 Minimal Lagrangian surfaces in $\CH,$  \cite{ML}$,$  \\
$(\maltese_{\ST})$ & $ -\ad (\r{P})  \overline
 {g(\bar \lambda)}^T,$ & 
 Timelike minimal Lagrangian surfaces in $\CHI,$ \cite{DK:timelike}$,$ \\
$(\bullet_{\SEA})$ & $\ad (I_*\r{P})  \> \overline
 {g(1/\bar \lambda)}, $&
 Elliptic or hyperbolic affine spheres in $\mathbb {R}^3,$ \cite{DW1}$,$ \\
 $(\maltese_{\SIA})$ & $\overline
 {g(\bar \lambda)},$&
 Indefinite affine spheres in $\R^3,$ \cite{DE}$,$
\end{tabular}

 where $I_{*}$ denotes $\id$ for the elliptic case $I_{2, 1}$ 
 for the hyperbolic case. 
 Then each of the classes of surfaces can be characterized 
 by some \textit{Tzitz\'eica equation}\footnote{The classical 
Tzitz\'eica equation is the one for the indefinite affine spheres.
 But also equations differing from the classical one by signs, like the equation above, are frequently called {\it Tzitz\'eica equation}.} :
\begin{align*}
(\bullet_{\SCP})\quad&  \omega^{\SCP}_{z \bar{z}} +  e^{\omega^{\SCPt}} -
 |Q^{\SCP}|^2 e^{-2 \omega^{\SCPt}} =0, \quad 
 Q^{\SCP}_{\bar{z}} = 0,   \\
(\bullet_{\SCH}) \quad& \omega^{\SCH}_{z \bar{z}} 
 -  e^{\omega^{\SCHt}} + |Q^{\SCH}|^2 e^{-2 \omega^{\SCHt}} =0, \quad 
 Q^{\SCH}_{\bar{z}} = 0, \\
 (\maltese_{\ST}) \quad& \omega^{\ST}_{uv} - e^{\omega^{\STt}} +e^{- 2 \omega^{\STt}} Q^{\ST}
 R^{\ST} =0,\quad  Q^{\ST}_{v} = R^{\ST}_u =0, \\
(\bullet_{\SEA}) \quad& \omega^{\SEA}_{z\bar{z}} + H e^{\omega^{\SEAt} } + |Q^{\SEA}|^2  e^{-2\omega^{\SEAt} } = 0, \;\;(H = \pm 1), \quad Q^{\SEA}_{\bar{z}}  = 0,
  \\
 (\maltese_{\SIA}) \quad& \omega^{\SIA}_{uv} - e^{\omega^{\SIAt}} +e^{- 2 \omega^{\SIAt}} 
 Q^{\SIA} R^{\SIA} =0,\quad \quad  Q^{\SIA}_{v} = R^{\SIA}_u =0.
\end{align*}
 Note that $Q^{\ST}, R^{\ST}$ take values in $i \R$ and 
 $Q^{\SIA}, R^{\SIA}$ take values in $\R$, respectively.

It is known that the above equations are different real forms of the $-1$-flow in the corresponding $A_2^{(2)}$-mKdV hierarchy, or the complex $A_2^{(2)}$-Toda field equation; and the real groups are exactly the automorphism groups of the corresponding geometries. 

 The fifth equation $(\maltese_{\SIA})$ has been studied in the context of 
 gas dynamics \cite{Ga} and pseudo-hyper-complex structures on 
 $\R^2 \times \R P^2$ \cite{Du}, and it is also related to harmonic 
 maps from $\R^{1,1}$ to the symmetric space ${\mathrm{SL}_3 \R / \mathrm{SO}_{2,1} \R}$. The fourth equation $(\bullet_{\SEA})$ above 
 can help construct semi-flat Calabi-Yau metrics and examples 
 for the SYZ Mirror Symmetry Conjecture, see \cite{LYZ, DP}. Specially the local radially symmetric solutions turn out to be Painlev\'e III transcendents. It is a striking universal feature of integrable systems that the same equation often arises from many unrelated sources. To further convince the reader of the great varieties here, we mention that minimal surfaces and Hamiltonian stationary Lagrangian surfaces
in $\CP$ and $\CH$ \cite{HR} also correspond to solutions of integrable systems
associated to $\sli_3 \C$, but with different automorphisms (of order $3$ and order $4$ respectively). 

One should also observe that in \cite{Kob} already all real forms of the affine algebra 
$A_1^{(1)}$ have been related to constant mean curvature/constant Gaussian curvature 
surfaces in the Euclidean $3$-space, the Minkowski $3$-space
 or the hyperbolic $3$-space.

The systematic construction from Lie theory above is just the starting point. It naturally gives rise to loop group factorizations, which in turn provide a method for constructing explicit solutions and symmetries of the equations. For example the classical B\"acklund and Darboux transformations have been generalized to dressing actions via loop group factorizations, see for examples Terng-Uhlenbeck \cite{TU} or Zakharov-Shabat \cite{ZS}. The classical Weierstrass representation of minimal surfaces has also been generalized by Dorfmeister-Pedit-Wu, \cite{DPW}, using Iwasawa type loop group factorizations. Many interesting questions naturally arise by translating between holomorphic/meromorphic data and properties of special geometric objects or special solutions of integrable PDEs. 
Although the original DPW method only considered  surfaces of conformal type (that is, associated with elliptic PDEs), it has also been generalized to surfaces of asymptotic line type (that is, associated with hyperbolic PDEs), such as 
 constant negative Gaussian curvature surfaces given by sine-Gordon equation, \cite{To}.
 Another way to get a very special class of solutions, called the finite type or finite gap solutions,  has beautiful and deep links to geometries of algebraic curves or Riemann surfaces and stable bundles over them, the so-called \textit{Hitchin systems}.

The paper is organized as follows: After discussing in the following sections one geometry for each real form of $A_2^{(2)}$ 
 we will compare their similarities and differences in Section \ref{sc:dpw} by the loop group method. To be self-contained and also to put this survey into a larger 
 context, we discuss the classification of our real forms in the last Section \ref{sc:real} from a geometric point of view.
 
%%%%%%%%%%%%%%%%%%%%%%%%%%%
\section{Minimal Lagrangian surfaces in $\CP$}\label{sc:mLi}
 In this section, we discuss a loop group formulation of 
 minimal Lagrangian surfaces in the complex projective plane $\CP$.
 The detailed discussion can be found in \cite{MM} or \cite{M}.
 In the following, the subscripts $z$ and $\bar z$ denote the derivatives with respect to 
 $z= x+ i y$ and $\bar z = x- i y$, respectively, that is,
\[
 f_z = \partial_z f := \frac{1}{2}\left( \frac{\partial f}{\partial x} - 
 i \frac{\partial f}{\partial y}\right), \quad
 f_{\bar z} = \partial_{\bar z} f := \frac{1}{2}\left( \frac{\partial f}{\partial x} +
 i \frac{\partial f}{\partial y}\right).
\]
\subsection{Basic definitions}
 We first consider the five-dimensional unit hypersphere $S^5$ 
 as a quadric in $\C^3$;
 \[
  S^5 = \{ v \in \C^3 \; | \; \langle v, v\rangle=1 \},
 \]
 where $\langle \;, \;\rangle$ is the standard Hermitian inner product in $\C^3$ 
 which is complex anti-linear in the second variable. 
 Then let $\CP$ be the two-dimensional complex projective plane and 
 consider the Hopf fibration $\pi : S^5 \to \CP$, given by $v \mapsto \C^{\times} v$. 
 We point out that the tangent space at 
 $u \in S^5$ is
\begin{equation*}
 T_u S^5  = \{ v \in \C^3 \;|\;\Re \langle v,u \rangle =0\}.
\end{equation*}
 Moreover, the space 
 $\mathcal{H}_u = \{ v \in T_u S^5\;|\;  \langle v,u \rangle =0\}$
 is a natural horizontal subspace.
The form $\langle \> ,\> \rangle$ is a positive definite Hermitian inner product on  
$\mathcal{H}_u$  with real and imaginary components
\[
 \langle \>,\> \rangle = g(\>,\>) + i \Omega(\>,\>).
\]
 Hence $g$ is positive definite and $\Omega$ is a symplectic form. 
 Put 
\[
 \U = \{ A : \C^3 \to \C^3 \;|\; \mbox{$\C$-linear satisfying 
 $\langle A u, A v\rangle =\langle u, v\rangle $}\},
\]
 and $\SU = \{ A \in \U \;| \; \det A =1\}$. We note $\U = S^1 \cdot 
 \SU$ and that these are connected real reductive Lie groups with their 
 centers consisting of multiples of the identity transformation.
 Then the groups $\U$ and $\SU$ act naturally on $S^5$ and $\CP$.
 The group $\U$ acts transitively on both spaces. Moreover, this action is 
 equivariant relative to $\pi$ and holomorphic on $\CP$. 
 Using the base point $e_3 = (0,0,1)^T$ it is easy to verify
 \[
  S^5 = \U/\Utwo \times \{1\}, \quad 
  \CP = \U/\Utwo \times S^1.
 \]
 
 %%%%%%%%%%%%%%%%%%%%%%%%%%%%
 
 \subsection{Horizontal lift and fundamental theorem}
 We now consider a Lagrangian immersion $f^{\SCP}$ 
 from a Riemann surface $M$ into $\CP$. 
 Then it is known that on an open and contractible subset $\D$
 of $M$, there exists a special lift into $S^5$, that is, 
 $\f^{\SCP} : \D \to S^5, \pi \circ \f^{\SCP} = f^{\SCP}|_{\D}$, 
 and 
\begin{equation}\label{eq:horizontal}
 \langle \d\f^{\SCP}, \f^{\SCP}\rangle=0
\end{equation}
 holds. The lift $\f^{\SCP}$ will be called 
 a \textit{horizontal lift} of $f^{\SCP}$.
 The induced metric of $f^{\SCP}$ is represented, 
 by using the horizontal lift $\f^{\SCP}$ as
 \[
  \d s^2 =  \Re \langle \d \f^{\SCP}, \d \f^{\SCP} \rangle.
 \]
 Since the induced metric is Riemannian, 
 we can assume that $f^{\SCP}$ is a conformal immersion 
 from $M$ to $\CP$.
 We take $z = x + i y$ to be its complex coordinates on 
 $\D \subset M$.
 Then the horizontality condition \eqref{eq:horizontal} implies 
 $\langle \f^{\SCP}_z, \f^{\SCP}\rangle
 = \langle \f^{\SCP}_{\bar z}, \f^{\SCP}\rangle=0$, 
 and taking the derivative with respect to 
 $\bar z$ of the first term and $z$ of the second term, respectively,
 we infer:
 \begin{equation}\label{eq:nullf}
 \langle \f^{\SCP}_{z}, \f^{\SCP}_{z}\rangle  = 
 \langle \f^{\SCP}_{\bar z}, \f^{\SCP}_{\bar z}\rangle>0.
\end{equation}
 Moreover, since $f^{\SCP}$ is conformal, we have
 \begin{equation}\label{eq:null}
 {\langle \f^{\SCP}_z, \f^{\SCP}_{\bar z}\rangle}  =  0.
\end{equation} 
 Therefore there exists a real function $\omega^{\SCP}: 
 \D \to \R$ such that 
\begin{equation}\label{eq:metric}
 \langle \f^{\SCP}_z, \f^{\SCP}_{z}\rangle  
 =\langle \f^{\SCP}_{\bar z}, \f^{\SCP}_{\bar z}\rangle  
 = e^{\omega^{\SCP}}, \quad \mbox{and}\quad 
 \d s^2 = 2 e^{\omega^{\SCP}} \d z  \d \bar z.
\end{equation}
 It is also easy to see from 
 $\langle \f^{\SCP}_x, \f^{\SCP}\rangle = \langle \f^{\SCP}_y, \f^{\SCP}\rangle = 0$, 
 and the derivative with respect to 
 $y$ of the first term and $x$ of the second term, respectively, that 
 \[
  \Omega (\f^{\SCP}_x, \f^{\SCP}_y)=0,
 \]
 that is,  $\f^{\SCP}$ is a Legendre immersion.
 We now consider the \textit{coordinate frame}
 \begin{equation}\label{eq:coordinateframeCP}
 \mathcal F_{\SCP} = ( e^{-\frac{1}{2}\omega^{\SCPt}}\f^{\SCP}_z, 
 e^{-\frac{1}{2}\omega^{\SCPt}} \f^{\SCP}_{\bar z}, \f^{\SCP}).
\end{equation}
 It is straightforward to see that $\mathcal F_{\SCP}$ takes values in 
 $\U$, that is, 
 $\overline{\mathcal F_{\SCP}}^T \, \mathcal F_{\SCP}  = 
 \id$.

For what follows it will be convenient to lift the mean curvature vector of $f^{\SCP}$ from
$T_{f^{\SCPt}(z)} \CP$ to $T_{\f^{\SCPt}(z)}  S^5$.
It is easy to verify that  the vectors $ \f^{\SCPt}_z, 
 \f^{\SCPt}_{\bar z}, i \f^{\SCPt}_{z}, i\f^{\SCPt}_{\bar z}, i\f^{\SCPt}$ span
 $(T_{\f^{\SCPt}(z)}  S^5)^{\C}$
and project under $\d \pi$ to $f^{\SCP}_z,f^{\SCP}_{\bar z},if^{\SCP}_z, 
 if^{\SCP}_{\bar z}, 0$ respectively.
In this sense we identify  the mean curvature vector 
$H = H_1 ie^{-\frac{1}{2}\omega^{\SCPt}} f^{\SCP}_z +  H_2 ie^{-\frac{1}{2}\omega^{\SCPt}} f^{\SCP}_{\bar z}$ of $f^{\SCP}$ with the vector
 $H = H_1 ie^{-\frac{1}{2}\omega^{\SCPt}} \f^{\SCP}_z  +  
 H_2 ie^{-\frac{1}{2}\omega^{\SCPt}} \f^{\SCP}_{\bar z}$.

\begin{Lemma}
 The coordinate frame $\mathcal{F}_{\SCP}$  of a Lagrangian immersion 
 into $\CP$ is a smooth map $\mathcal{F}_{\SCP}: \D \rightarrow \U$.
 In particular, $\det \mathcal{F}_{\SCP}$ is a smooth map  
 from $ \D $ to $S^1$.
 The Maurer-Cartan form 
\begin{equation}\label{eq:alphaCP}
 \alpha_{\SCP} = \mathcal F_{\SCP}^{-1} \d \mathcal F_{\SCP} = 
 \mathcal F_{\SCP}^{-1} 
 ({\mathcal F}_{\SCP})_z \d z
 +  \mathcal F_{\SCP}^{-1} (\mathcal F_{\SCP})_{\bar z} \d \bar z 
 =  \mathcal U_{\SCP}  \d z  
+  \mathcal V_{\SCP} \d \bar z 
\end{equation}
 can be computed as 
\begin{align}
 \mathcal U_{\SCP} = \begin{pmatrix} 
 \frac{1}{2}\omega^{\SCP}_z + \ell  & m & e^{\frac{1}{2}\omega^{\SCPt}}\\
  -Q^{\SCP} e^{-\omega^{\SCPt}} &- \frac{1}{2}\omega^{\SCP}_z+ \ell & 0\\
 0 & - e^{\frac{1}{2}\omega^{\SCPt}}& 0
 \end{pmatrix}, \quad 
 \mathcal V_{\SCP} =
 \begin{pmatrix} 
 - \frac{1}{2}\omega^{\SCP}_{\bar z}+ m
 & \overline {Q^{\SCP}  e^{-\omega^{\SCPt}}}  & 0 \\
 \ell & \frac{1}{2}\omega^{\SCP}_{\bar z} + m& e^{\frac{1}{2}\omega^{\SCPt}}\\
 - e^{\frac{1}{2}\omega^{\SCPt}}& 0 & 0
\end{pmatrix}, 
 \label{eq:UVCP}
\end{align}
 where $\ell = \langle H, \f^{\SCP}_{\bar z} \rangle$, 
 $m =  \langle H, \f^{\SCP}_z \rangle $, $H$ denotes
 the mean curvature vector, and $Q^{\SCP}$ is defined by
\begin{equation}\label{eq:QCP}
 Q^{\SCP} = \langle \f^{\SCP}_{zzz}, \f^{\SCP}\rangle.
\end{equation}
 Here we have used $\langle H, \f^{\SCP}_{\bar z} \rangle
 = - \langle \f^{\SCP}_{z}, H \rangle$ and 
 $\langle H, \f^{\SCP}_{z} \rangle
 = - \langle \f^{\SCP}_{\bar z}, H \rangle$.
Moreover, $m = - \bar \ell$ holds.
\end{Lemma}
\begin{Corollary}
 For $\alpha_{\SCP}$ in \eqref{eq:alphaCP}, 
 the following statements hold, see for example  \cite[Section 2.1]{M}$:$
\begin{enumerate}
\item The {\rm mean curvature 1-form} $\sigma_H^{\SCP} 
 =  \Omega(H, \d \f^{\SCP})$
  satisfies $i \sigma_H^{\SCP} = \langle H, \d\f^{\SCP} \rangle 
 = \frac{1}{2} {\rm trace} ( \alpha_{\SCP})$.
\item  The $\alpha_{\SCP}$ satisfies the Maurer-Cartan equations 
 if and only if
 \begin{align}\label{eq:FundaeqCP}
&\omega^{\SCP}_{z \bar z}+ \left(1+ \frac{1}{2}|H|^2\right) e^{\omega^{\SCPt}} -|Q^{\SCP}|^2e^{- 2 \omega^{\SCPt}}  =0, \\\label{eq:FundaeqCP2}
 &\d \sigma^{\SCP}_H=0, \quad Q^{\SCP}_{\bar z} 
 e^{-2 \omega^{\SCPt}}= - (\ell e^{- \omega^{\SCPt}})_z.
\end{align}
\end{enumerate}
\end{Corollary}

 Then the fundamental theorem for Lagrangian immersions 
 into $\CP$ is stated as follows:

 \begin{Theorem}[Fundamental theorem for Lagrangian immersions into
 $\CP$]\label{thm:fundCP}
 Assume $f^{\SCP}:\D \rightarrow  \CP$ is a conformal 
 Lagrangian immersion and let $\f^{\SCP}$ denote one of 
 its horizontal lifts and $\mathcal{F}_{\SCP}$ 
 the corresponding coordinate frame \eqref{eq:coordinateframeCP}.
 Then $\alpha_{\SCP} = \mathcal{F}_{\SCP}^{-1} 
 \d \mathcal{F}_{\SCP} 
 = \mathcal{U}_{\SCP} \d z + \mathcal{V}_{\SCP} \d \bar{z}$ 
 with $\mathcal{U}_{\SCP}$ and $\mathcal{V}_{\SCP}$ 
 have the form \eqref{eq:UVCP} 
 and their coefficients satisfy the equations stated 
 in \eqref{eq:FundaeqCP} and \eqref{eq:FundaeqCP2}.
 
 Conversely, given  functions $\omega^{\SCP}, H$ on $\D$ together with 
 a cubic differential $Q^{\SCP}\d z^3$ and a $1$-form $\sigma^{\SCP}_H = 
 \ell \d z  + m \d \bar{z}$ on $\D$ such that the conditions 
 \eqref{eq:FundaeqCP} and \eqref{eq:FundaeqCP2} are 
 satisfied $($with $\langle H, \f^{\SCP}_{\bar{z}} \rangle$  
 replaced by $m$  and $\langle H, \f^{\SCP}_z \rangle$ 
replaced by $\ell$$)$,
 then there exists a solution $\mathcal{F}_{\SCP} \in \U$ 
 such that 
 $\f^{\SCP} = \mathcal{F}_{\SCP} e_3$ is a horizontal lift 
 of the conformal 
 Lagrangian immersion $f^{\SCP} = \pi \circ \f^{\SCP}$.
\end{Theorem}

\subsection{Minimal Lagrangian surfaces in $\CP$}
 If we restrict to minimal Lagrangian surfaces, then the equations 
 (\ref{eq:alphaCP}) and (\ref{eq:UVCP}) show that the determinant of the coordinate 
 frame is a constant (in $S^1$).
 So we can, and will, assume from here on that the horizontal lift of the given minimal immersion into $\CP$ is scaled  (by a constant in $S^1$) such that the corresponding coordinate frame $\mathcal{F}_{\SCP}$ is in $\SU$.
 It is clear that the Maurer-Cartan form $ \alpha_{\SCP} = \mathcal F_{\SCP}^{-1} \d \mathcal F_{\SCP} = \mathcal U_{\SCP} \d z + \mathcal V_{\SCP} \d  \bar z$ of the minimal 
 Lagrangian surface 
 is given by  
\begin{align}
 \mathcal U_{\SCP} = \begin{pmatrix} 
 \frac{1}{2}\omega^{\SCP}_z   & 0 & e^{\frac{1}{2}\omega^{\SCPt}}\\
  -Q^{\SCP} e^{-\omega^{\SCPt}}&- \frac{1}{2}\omega^{\SCP}_z & 0\\
 0 & - e^{\frac{1}{2}\omega^{\SCPt}}& 0
 \end{pmatrix}, \quad 
 \mathcal V_{\SCP} =
 \begin{pmatrix} 
 - \frac{1}{2}\omega^{\SCP}_{\bar z}
 &  \overline {Q^{\SCP}} e^{-\omega^{\SCPt}}  & 0 \\
 0 & \frac{1}{2}\omega^{\SCP}_{\bar z} & e^{\frac{1}{2}\omega^{\SCPt}}\\
 - e^{\frac{1}{2}\omega^{\SCPt}}& 0 & 0
\end{pmatrix}, 
 \label{eq:MLUVCP}
\end{align}
 and 
 the integrability conditions are 
\begin{equation} \label{eq:TzitzeicaCP}
 \omega^{\SCP}_{z \bar{z}} +  e^{\omega^{\SCPt}} -
 |Q^{\SCP}|^2 e^{-2 \omega^{\SCPt}} =0, \quad 
 Q^{\SCP}_{\bar{z}} = 0.
 \end{equation}
  The first equation \eqref{eq:TzitzeicaCP} is commonly called the 
  \textit{Tzitz\'eica equation}.
  From the definition of $Q^{\SCP}$ in \eqref{eq:QCP}, it is clear that 
 \[
  C^{\SCP}(z) = Q^{\SCP}(z) \, \d z^3 
 \]
 is the holomorphic cubic differential of the  
 minimal Lagrangian surface $f^{\SCP}$.

\begin{Remark}
 The fundamental theorem in Theorem \ref{thm:fundCP} is 
 still true for a minimal Lagrangian immersions into $\CP$.
 \end{Remark}

 %%%%%%%%%%%%%%%%%%%%%%%%%%
 \subsection{Associated families of minimal surfaces and 
 flat connections}
 From \eqref{eq:TzitzeicaCP}, it is easy to see that there exists a one-parameter 
 family of solutions of \eqref{eq:TzitzeicaCP} 
 parametrized by $\l \in S^1$; 
 The corresponding family $\{ {\omega_{\SCPt}^\l}, 
 C_{\SCP}^\l\}_{\l \in S^1}$
 then satisfies
\[
  \omega_{\SCP}^\l = \omega^{\SCP}, \quad C_{\SCP}^{\l} = 
 \l^{-3} Q^{\SCP} \d z^3.
\]
 As a consequence, there exists a one-parameter family of 
 minimal Lagrangian surfaces 
 $\{\hat f_{\SCP}^{\l}\}_{\l \in S^1}$ such that $\hat f_{\SCP}^{\l}
 |_{\l=1} = f^{\SCP}$.
 The family $\{\hat f_{\SCP}^{\l}\}_{\l \in S^1}$ will be called 
 the \textit{associated 
 family} of $f^{\SCP}$.
 Let $\hat {\mathcal F}^{\l}_{\SCP}$ 
 be the coordinate frame of $\hat f_{\SCP}^{\l}$.
 Then the Maurer-Cartan form $\hat \alpha_{\SCP}^{\l}
 = \hat {\mathcal U}^{\l}_{\SCP}\d z + 
 \hat {\mathcal V}^{\l}_{\SCP} \d \bar z$ of 
 $\hat {\mathcal F}^{\l}_{\SCP}$
 for the associated family 
 $\{\hat f_{\SCP}^{\l}\}_{\lambda \in S^1}$
 is given by ${\mathcal U}_{\SCP}$ and ${\mathcal V}_{\SCP}$ 
 as in 
 \eqref{eq:MLUVCP} where we have replaced 
 $Q^{\SCP}$ and $\overline{ Q^{\SCP}}$ by 
 $\lambda^{-3} Q^{\SCP}$ and $\lambda^3 \overline{Q^{\SCP}}$, 
 respectively.
 Then consider the gauge transformation $G^{\l}$ given by 
 \begin{equation}\label{eq:extCP}
 F^{\l}_{\SCP}  =\hat {\mathcal F}^{\l}_{\SCP} G^{\l}, \quad 
 G^{\l} = \di ( \lambda, \lambda^{-1}, 1).
 \end{equation}
 This implies
\begin{equation}\label{eq:alphaL}
\alpha_{\SCP}^{\l} =  (F^{\l}_{\SCP})^{-1}  \d F^{\l}_{\SCP} 
 = U_{\SCP}^{\l} \d z + V_{\SCP}^{\l} \d \bar z	   
\end{equation}
 with $U_{\SCP}^{\l} = (G^{\l})^{-1}\hat{\mathcal U}^{\l}_{\SCP} G^{\l}$ 
 and $V_{\SCP}^{\l} = (G^{\l})^{-1}\hat{\mathcal V}^{\l}_{\SCP} G^{\l}$. 
 It is easy to see that $\hat{\mathcal F}^{\l}_{\SCP} G^{\l} e_3= 
 \hat{\mathcal F}^{\l}_{\SCP} e_3$. 
 Therefore 
 $f_{\SCP}^{\l} : = \pi \circ (\hat{\mathcal F}^{\l}_{\SCP} G^{\l}e_3) 
 =  \pi \circ (\hat{\mathcal F}^{\l}_{\SCP} e_3)  = \hat{f}_{\SCP}^\l$. 
 Hence we will not distinguish between 
 $\{\hat f_{\SCP}^{\lambda}\}_{\lambda \in S^1}$ and 
 $\{f_{\SCP}^{\lambda}\}_{\lambda \in S^1}$, and both families will be 
 called the associated family of $f^{\SCP}$, and $F_{\SCP}^{\l}$ 
 will also  be called the coordinate frame of $f_{\SCP}^{\l}$.
  
 From the discussion just above we derive a 
 family of Maurer-Cartan forms $\alpha_{\SCP}^{\l}$  
 in \eqref{eq:alphaL}
 of  minimal Lagrangian surfaces  from $\D$  to $\CP$ . 
 They can be computed explicitly as 
  \begin{equation}\label{eq:alphalambda-origCP}
 \alpha_{\SCP}^{\l} = U_{\SCP}^{\lambda} \d z + V_{\SCP}^{\lambda} \d\bar z,
 \end{equation} 
 for $\lambda \in \C^{\times}$, where 
 $U^{\SCP}_{\lambda}$ and $V_{\SCP}^{\lambda}$ are  given by  
\begin{align*}
 U_{\SCP}^{\lambda} =
 \begin{pmatrix} 
\frac{1}{2}{\omega^{\SCP}_z} & 0 &
 \lambda^{-1} e^{\frac{1}{2}\omega^{\SCPt}}\\
- \lambda^{-1} Q^{\SCP} e^{-\omega^{\SCPt}} & 
 - \frac{1}{2}\omega^{\SCP}_z & 0\\
0 & -\lambda^{-1}e^{\frac{1}{2}\omega^{\SCPt}}& 0
\end{pmatrix},   \quad 
V_{\SCP}^{\lambda} =
 \begin{pmatrix} 
- \frac{1}{2}\omega_{\bar z}
 &  \lambda \overline{ Q^{\SCP}} e^{-\omega^{\SCPt}} & 0 \\
 0 & \frac{1}{2}\omega^{\SCP}_{\bar z} & \lambda e^{\frac{1}{2}\omega^{\SCP}}\\
-\lambda e^{\frac{1}{2}\omega^{\SCPt}}& 0 & 0
\end{pmatrix}. 
\end{align*}
 It is clear that $\alpha_{\SCP}^{\lambda}|_{\lambda=1}$ is 
 the Maurer-Cartan form of the coordinate frame 
 $\mathcal F_{\SCP}$ of $f^{\SCP}$. 
 Then by the discussion in the previous section, we have the following 
 theorem. 
\begin{Theorem}[\cite{MM}]\label{thm:flatconnectionsCP}
 Let $f^{\SCP} : \D \to \CP$ be a minimal Lagrangian surface in $\CP$
 and let $\alpha_{\SCP}^{\l}$ be the family of Maurer-Cartan 
 forms defined in \eqref{eq:alphalambda-origCP}.
 Then $\d + \alpha_{\SCP}^{\lambda}$ gives a family of flat connections
 on $\D \times \SU$.

 Conversely, given a family of connections  $\d 
 + \alpha_{\SCP}^{\lambda}$ 
 on $\D \times \SU$, where $\alpha_{\SCP}^{\lambda}$ 
 is as in \eqref{eq:alphalambda-origCP}, then  
 $\d + \alpha_{\SCP}^{\lambda}$ belongs 
 to an associated 
 family of minimal Lagrangian immersions into $\CP$ 
 if and only if the connection is flat for all 
 $\lambda \in S^1$.
\end{Theorem}

%
%%%%%%%%%%%%%%%%%%%%%%%%%%%%%%
\section{Minimal Lagrangian surfaces in $\CH$}\label{sc:mLiCH}

In this section, we discuss a loop group formulation of minimal Lagrangian surfaces in the complex hyperbolic plane $\CH$. Most of what we present can be found in \cite{ML}.
We will use complex parameters and restrict generally to surfaces defined on some 
open and simply connected domain $\D$ of the complex plane $\C$.

%%%%%%%%%%%%%%%%%%%%%%%%%%%%
\subsection{Basic definitions}
 The space  $\CH$ can be realized as the open unit disk in 
 $\C^2$ relative to the usual positive definite Hermitian inner product. 
 But for our purposes it will be more convenient to realize 
 $\CH$ in the form
\begin{equation*}
\CH = \{ [w_1, w_2, 1] \in \CP \;|\; |w_1|^2 + |w_2|^2 - 1 < 0 \}.
\end{equation*}
 It is natural then to consider on $\C^3_1$ 
 the indefinite Hermitian inner form of signature $(1, 2)$
 given by 
\begin{equation}
\langle u,v \rangle = u_1 \bar{v}_1 +  u_2 \bar{v}_2 -  u_3 \bar{v}_3.
\end{equation}
Vectors in $\C^3_1$ satisfying  $\langle u, u \rangle < 0$ 
 will be called ``negative''.
 Then the set $(\C^3_1)_-$ of negative vectors and the 
 ``negative sphere''
\begin{equation}
 H^5_1 = \{ u \in \C^3_1\;|\; \langle u,u\rangle = -1\},
\end{equation}
 and the natural (submersions) projections 
 $\pi: (\C^3_1)_- \rightarrow \CH$  and  $\pi: H^5_1 \rightarrow \CH$ 
 will be the central objects of this section. 
 (Note that we use the same letter for both projections.)
This is called the Boothby-Wang type fibration, \cite{BW, DF}.
 For later purposes we point out that the tangent space at 
 $u \in H_1^5$ is
\begin{equation*}
 T_u H_1^5  = \{ v \in \C^3_1\;|\;\Re \langle v,u \rangle =0\}.
\end{equation*}
 Moreover, the space 
 $\mathcal{H}_u = \{ v \in T_u H_1^5\;|\;  \langle v,u \rangle =0\}$
 is a natural horizontal subspace.
The form $\langle \> ,\> \rangle$ is a positive definite Hermitian inner product on  
$\mathcal{H}_u$  with real and imaginary components
\[
 \langle \>,\> \rangle = g(\>,\>) + i \Omega(\>,\>).
\]
 Hence $g$ is positive definite and $\Omega$ is a symplectic form.
 Clearly, the isometry group of $\langle \>,\> \rangle$ 
 will be of importance in our  setting.
 Put
\begin{equation*}
\UI = \{ A: \C^3_1 \rightarrow \C^3_1\;|\;  \mbox{$\C$-linear satisfying} \hspace{1mm}\langle Au,Av \rangle = \langle u,v \rangle \}, 
 \end{equation*}
 and $\SUI = \{A \in \UI\;|\; \det A = 1 \}$.
 We note $\UI = S^1 \cdot \SUI$ and that these are connected, real, reductive Lie groups with their centers consisting of multiples of the identity transformation. 

The groups $\UI$ and  $\SUI$ act naturally  on $H_1^5$ and on $\CH$. 
The group $\UI$ acts transitively on both spaces.
Moreover,  this action is equivariant relative to $\pi$
and holomorphic on $\CH$. 
Using the base point $e_3 = (0,0,1)^T$ it is easy to verify
\[
 H_1^5 \cong \UI / {\Utwo \times \{1\}}
 \quad \mbox{and} \quad 
 \CH \cong \UI/{\Utwo \times S^1}.
\]
%%%%%%%%%%%%%%%%%%%%%
\subsection{Horizontal lift and fundamental theorem}
 We now consider a Lagrangian immersion $f^{\SCH}$ 
 from a Riemann surface $M$ into $\CH$. 
 Then it is known that on an open and contractible subset $\D$
 of $M$, there exists a special lift into $H_1^5$, that is, 
 $\f^{\SCH}: \D \rightarrow  H_1^5$ such that 
 $f^{\SCH}|_{\D} = \pi \circ \f^{\SCH}$ holds.
 Without loss of generality the lift $\f^{\SCH}$ satisfies
\begin{equation} \label{horizontal}
\langle \d\f^{\SCH},\f^{\SCH} \rangle = 0,
\end{equation}
 and it is called a \textit{horizontal lift}. Moreover, any two 
 such horizontal lifts
 only differ by a constant multiplicative factor from $S^1$.

 From equation \eqref{horizontal} we obtain  
 $\langle \f^{\SCH}_z,\f^{\SCH} \rangle = 0 = \langle 
 \f^{\SCH}_{\bar{z}},\f^{\SCH} \rangle$ 
 and after differentiation for $\bar{z}$ and $z$ 
 respectively we derive
 $\langle \f^{\SCH}_z,\f^{\SCH}_z  \rangle = \langle 
 \f^{\SCH}_{\bar{z}},\f^{\SCH}_{\bar{z}}  \rangle = e^{\omega^{\SCHt}}$ for some real function $\omega^{\SCH}: \D
 \to \R$.
  It is also easy to see from 
 $\langle \f^{\SCH}_x, \f^{\SCH} \rangle = \langle \f^{\SCH}_y, \f^{\SCH} \rangle = 0$, 
 and the derivative with respect to 
 $y$ of the first term and $x$ of the second term, respectively, that 
 \[
  \Omega (\f^{\SCH}_x, \f^{\SCH}_y)=0,
 \]
 that is,  $\f^{\SCH}$ is a Legendre immersion.
 Moreover, since $f^{\SCH}$ is conformal, 
 we also have $\langle \f^{\SCH}_z,\f^{\SCH}_{\bar{z}}  \rangle =0.$
 Therefore the metric of $f^{\SCH}$ is given by
 \[
 \d s^2 = \Re \langle \d \f^{\SCH}, \d \f^{\SCH}\rangle 
 =  2 e^{\omega^{\SCHt}} \d z \d \bar z.
 \]
 As a consequence, the vectors $ e^{-\omega^{\SCHt}/2} \f_z$, 
 $e^{-\omega^{\SCHt}/2} \f_{\bar{z}}$ and $\f$ form an 
 ``orthonormal basis'' relative to our Hermitian inner product
 $\langle \>, \>\rangle$. 
 Let us consider the coordinate frame
\begin{equation} \label{eq:coordframeCH}
\mathcal{F}_{\SCH} = ( e^{-\frac{1}{2}\omega^{\SCHt}} 
 \f^{\SCH}_z, e^{-\frac{1}{2}\omega^{\SCHt}} \f^{\SCH}_{\bar{z}},\f^{\SCH}).
\end{equation}

For what follows it will be convenient to lift the mean curvature vector of $f^{\SCH}$ from
$T_{f^{\SCHt}(z)} \CH$ to $T_{\f^{\SCHt}(z)}  H^5_1$.
It is easy to verify that  the vectors $ \f^{\SCH}_z, \f^{\SCH}_{\bar z},
  i \f^{\SCH}_z, i\f^{\SCH}_{\bar z}, i\f^{\SCH}$ span 
 $(T_{\f^{\SCH}(z)}  H^5_1)^{\C}$
and project under $\d \pi$ to $f^{\SCH}_z,f^{\SCH}_{\bar z},
 if^{\SCH}_z, if^{\SCH}_{\bar z}, 0$ respectively.
In this sense we identify  the mean curvature vector 
$H = H_1 ie^{-\frac{1}{2}\omega^{\SCHt}} f^{\SCH}_z +  H_2 ie^{-\frac{1}{2}\omega^{\SCHt}} f^{\SCH}_{\bar z}$ of $f^{\SCH}$ with the vector
 $H = H_1 ie^{-\frac{1}{2}\omega^{\SCHt}} \f^{\SCH}_{z}  +  
 H_2 ie^{-\frac{1}{2}\omega^{\SCHt}} \f^{\SCH}_{\bar z}$.
 It is clear now that we have the following, see \cite{ML}$:$
\begin{Lemma}
The coordinate frame $\mathcal{F}_{\SCH}$  of a Lagrangian immersion 
 into $\CH$ is a smooth map $\mathcal{F}_{\SCH}: \D \rightarrow \UI$.
 In particular, $\det \mathcal{F}_{\SCH}$ is a smooth map  
 from $ \D $ to $S^1$.
 For the Maurer-Cartan form 
\begin{equation}\label{eq:alphaH}
\alpha_{\SCH} = 
 \mathcal{F}_{\SCH}^{-1} \d \mathcal{F}_{\SCH} = 
 \mathcal{U}_{\SCH} \d z + \mathcal{V}_{\SCH} \d \bar{z},
\end{equation} 
 one then obtains, 
\begin{align}
 \mathcal U_{\SCH} = \begin{pmatrix} 
 \frac{1}{2}\omega^{\SCH}_z + \ell  & m & e^{\frac{1}{2}\omega^{\SCHt}}\\
  -Q^{\SCH}e^{-\omega^{\SCH}}&- \frac{1}{2}{\omega}^{\SCH}_z + \ell 
 & 0\\ 0 & e^{\frac{1}{2}\omega^{\SCH}}& 0
 \end{pmatrix}, \;\;
 \mathcal V_{\SCH} =
 \begin{pmatrix} 
 - \frac{1}{2} \omega^{\SCH}_{\bar{z}}  + m & \overline{Q^{\SCH}} 
 e^{-\omega^{\SCHt}}  & 0 \\
 \ell  & \frac{1}{2}\omega^{\SCH}_{\bar{z}} +m & 
 e^{\frac{1}{2}\omega^{\SCH}}\\
 e^{\frac{1}{2}\omega^{\SCH}}& 0 & 0
\end{pmatrix}, 
 \label{eq:UVH2}
\end{align}
 where  $\ell =  \langle H, \f^{\SCH}_{\bar z} \rangle$,
 $m =  \langle H, \f^{\SCH}_z \rangle$ and
 $H$ denotes the mean curvature vector. Moreover we have
\begin{equation}\label{eq:QCH}
Q^{\SCH} = \langle \f^{\SCH}_{zzz}, \f^{\SCH} \rangle.
\end{equation} 
 Here we have used $ \langle H, \f^{\SCH}_{\bar{z}} \rangle = 
- \langle \f^{\SCH}_z , H \rangle$ and $\langle H, \f^{\SCH}_{z} 
 \rangle = 
- \langle \f^{\SCH}_{\bar{z}},H \rangle$. Moreover $m = - \bar \ell$ holds.
\end{Lemma}

\begin{Corollary}
 \label{MCH2}
 For $\alpha_{\SCH}$ in \eqref{eq:alphaH}, 
 the following statements hold see for example \cite{ML}$:$
\begin{enumerate}
 \item The {\rm mean curvature 1-form} $\sigma_H^{\SCH} 
 =  \Omega(H, \d \f^{\SCH}) = \ell \d z + m \d \bar z$
  satisfies $i \sigma_H^{\SCH} = \langle H, \d\f^{\SCH} \rangle 
 = \frac{1}{2} {\rm trace} ( \alpha_{\SCH})$.

\item The $1$-form $\alpha_{\SCH}$ satisfies the Maurer-Cartan equations 
 if and only if
 \begin{align*}
& \omega^{\SCH}_{z \bar{z}} - \left(1- \frac{1}{2}|H|^2\right) 
 e^{\omega^{\SCHt}} - 
 |Q^{\SCH}|^2 e^{-2 \omega^{\SCHt}} =0, \\
& \d \sigma^{\SCH}_H =0, \quad 
 Q^{\SCH}_{\bar{z}} e^{-2\omega^{\SCHt}} = 
 -  (\ell e^{-\omega^{\SCH}})_z.
\end{align*}
\end{enumerate}
\end{Corollary}

 From this one obtains the following theorem.
 \begin{Theorem}[Fundamental theorem for Lagrangian immersions into
 $\CH$]\label{thm:fundCH}
 Assume $f^{\SCH}:\D \rightarrow  \CH$ is a conformal 
 Lagrangian immersion and let $\f^{\SCH}$ denote one of 
 its horizontal lifts and $\mathcal{F}_{\SCH}$ 
 the corresponding coordinate frame \eqref{eq:coordframeCH}.
 Then $\alpha_{\SCH} = (\mathcal{F}_{\SCH})^{-1} \d \mathcal{F}_{\SCH} 
 = \mathcal{U}_{\SCH} \d z + \mathcal{V}_{\SCH} \d \bar{z}$ 
 with $\mathcal{U}_{\SCH}$ and $\mathcal{V}_{\SCH}$ 
 having the form \eqref{eq:UVH2} 
 and their coefficients satisfying the equations stated 
 in Corollary {\rm \ref{MCH2}}.
 
 Conversely, given  functions $\omega^{\SCH}, H$ on $\D$ together with 
 a cubic differential $Q^{\SCH}\d z^3$ and a $1$-form $\sigma^{\SCH}_H = 
 \ell \d z  + m \d \bar{z}$ on $\D$ such that the conditions of 
 Corollary  {\rm \ref{MCH2}}
 are satisfied $($with $\langle H, \f^{\SCP}_{\bar{z}} \rangle$  replaced by $m$  and 
 $\langle H, \f^{\SCP}_z \rangle$ 
 replaced by $\ell$$)$, 
 then there exists a solution $\mathcal{F}_{\SCH} \in \UI$ such that 
 $\f^{\SCH} = \mathcal{F}_{\SCH} e_3$ is a horizontal lift of the conformal 
 Lagrangian immersion $f^{\SCH} = \pi \circ \f^{\SCH}$.
\end{Theorem}

\subsection{Minimal Lagrangian surfaces in $\CH$}
If we restrict to minimal Lagrangian surfaces, then 
$\ell$ and $m$ vanish identically. Moreover, the equations (\ref{eq:UVH2}) show that the determinant of the coordinate frame is a constant (in $S^1$).
So we can, and will, assume from here on that the horizontal lift of the given minimal immersion into $\CH$ is scaled  (by a constant in $S^1$) such that the corresponding coordinate frame $\mathcal{F}_{\SCH}$ is in $\SUI$.
 It follows that the matrices in \eqref{eq:UVH2}  now are of the form 
\begin{align}
 \mathcal U_{\SCH} = \begin{pmatrix} 
 \frac{1}{2}\omega^{\SCH}_z  & 0 & e^{\frac{1}{2}\omega^{\SCHt}}\\
  - Q^{\SCH}e^{-\omega^{\SCH}} &- \frac{1}{2}{\omega}^{\SCH}_z 
 & 0\\[0.1cm] 0 & e^{\frac{1}{2}\omega^{\SCH}}& 0
 \end{pmatrix}, \;\;
 \mathcal V_{\SCH} =
 \begin{pmatrix} 
 - \frac{1}{2} \omega^{\SCH}_{\bar{z}}  & \overline{Q^{\SCH}} 
 e^{-\omega^{\SCHt}}  & 0 \\
 0 & \frac{1}{2}\omega^{\SCH}_{\bar{z}}& 
 e^{\frac{1}{2}\omega^{\SCH}}\\
 e^{\frac{1}{2}\omega^{\SCH}}& 0 & 0
\end{pmatrix}, 
 \label{eq:UVH2min}
\end{align}
and the integrability conditions are 
\begin{equation} \label{integrmin}
 \omega^{\SCH}_{z \bar{z}} -  e^{\omega^{\SCHt}} + 
 |Q^{\SCH}|^2 e^{-2 \omega^{\SCHt}} =0, \hspace{6mm}
 Q^{\SCH}_{\bar{z}} = 0.
 \end{equation}
 Note, the first of these two equations is 
 one of the  Tzitz\'eica equations
 (which differ from each other by some sign(s)).
 From the definition of $Q^{\SCH}$ in \eqref{eq:QCH}, it is clear that 
 \[
  C^{\SCH}(z) = Q^{\SCH}(z) \, \d z^3 
 \]
 is the holomorphic cubic differential of the  
 minimal Lagrangian surface $f^{\SCH}$.

\begin{Remark}
 The fundamental theorem in Theorem \ref{thm:fundCH} is 
 still true for a minimal Lagrangian immersions into $\CH$.
 \end{Remark}

%%%%%%%%%%%%%%%%%%%%%%%
\subsection{Associated families and flat connections}
 From \eqref{integrmin}, it is easy to see that there 
 exists a one-parameter 
 family of solutions of \eqref{integrmin}
 parametrized by $\l \in S^1$.
 The corresponding family $\{\omega_{\SCH}^\l, 
 C_{\SCH}^{\l} \}_{\l \in S^1}$
 then satisfies
\[
  \omega_{\SCH}^{\l} = \omega^{\SCH}, \quad C_{\SCH}^{\l} = \l^{-3} Q^{\SCH} \d z^3.
\]
 As a consequence, there exists a one-parameter family of 
 minimal Lagrangian surfaces 
 $\{\hat f_{\SCH}^{\l}\}_{\l \in S^1}$  in $\CH$ such that $\hat f_{\SCH}^{\l}|_{\l=1} = f^{\SCH}$.
 The family $\{\hat f_{\SCH}^{\l}\}_{\l \in S^1}$ will be called the \textit{associated 
 family} of $f^{\SCH}$.
 Let $\hat {\mathcal F}^{\l}_{\SCH}$ be the coordinate frame of 
 $\hat f_{\SCH}^{\l}$.
 Then the Maurer-Cartan form $\hat \alpha_{\SCH}^{\l} = 
 \hat {\mathcal U}_{\SCH}^{\l} \d z + 
 \hat {\mathcal V}_{\SCH}^{\l} \d \bar z$ of $\hat {\mathcal F}_{\SCH}^{\l}$
 for the associated family 
 $\{\hat f_{\SCH}^{\l}\}_{\lambda \in S^1}$
 is given by ${\mathcal U}_{\SCH}$ and ${\mathcal V}_{\SCH}$ 
 as in 
 \eqref{eq:UVH2min} where we have replaced $Q^{\SCH}$ and $\overline{Q^{\SCH}}$ by $\lambda^{-3} Q^{\SCH}$ 
 and $\lambda^3 \overline{Q^{\SCH}}$, respectively.
 Then consider the gauge transformation $G^{\l}$ given by 
 \begin{equation}\label{eq:extCH}
 F_{\SCH}^{\l}  =\hat {\mathcal F}^{\l}_{\SCH} G^{\l}, \quad 
 G^{\l} = \di ( \lambda, \lambda^{-1}, 1).
 \end{equation}
 This implies
\begin{equation} \label{eq:alphaLH2}
\alpha_{\SCH}^{\l} =  (F_{\SCH}^{\l})^{-1}  \d F_{\SCH}^{\l} 
 = U_{\SCH}^{\l} \d z + V_{\SCH}^{\l} \d \bar z	   
\end{equation} \label{alphaLH2}
 with $U_{\SCH}^{\l} = (G^{\l})^{-1}\hat{\mathcal U}^{\l}_{\SCH} G^{\l}$ 
 and $V_{\SCH}^{\l} = (G^{\l})^{-1}\hat{\mathcal V}^{\l}_{\SCH} G^{\l}$. 
 It is easy to see that $\hat{\mathcal F}_{\SCH}^{\l} G^{\l} e_3
 = \hat{\mathcal F}_{\SCH}^{\l} e_3$. 
 Therefore 
 $f_{\SCH}^{\l} : = \pi \circ (\hat{\mathcal F}_{\SCH}^{\l} G^{\l}e_3) =  
 \pi \circ (\hat{\mathcal F}_{\SCH}^{\l} e_3)  = \hat{f}_{\SCH}^\l$. 
 Hence we will not distinguish between 
 $\{\hat f_{\SCH}^{\lambda}\}_{\lambda \in S^1}$ and 
 $\{f_{\SCH}^{\lambda}\}_{\lambda \in S^1}$, and both families will be 
 called the associated family of $f^{\SCH}$, and $F_{\SCH}^{\l}$ will also  be called the 
 coordinate frame of $f_{\SCH}^{\l}$.
 
 %%%%%%%%%%%%%%%%%%%%%%%

 From the discussion just above we obtain that 
 the family of Maurer-Cartan forms $\alpha_{\SCH}^{\l}$  
 in \eqref{eq:alphaLH2}
 of a minimal Lagrangian surface $f^{\SCH}: M \to \CP$  can be 
 computed explicitly as 
  \begin{equation}\label{eq:alphalambda-origCH}
 \alpha_{\SCH}^{\l} = U_{\SCH}^{\lambda} \d z + V_{\SCH}^{\lambda} 
 \d\bar z,
 \end{equation} 
 for $\lambda \in \C^{\times}$, where 
 $U_{\SCH}^{\lambda}$ and $V_{\SCH}^{\lambda}$ are  given by  
\begin{align*}
U_{\SCH}^{\lambda} =
 \begin{pmatrix} 
\frac{1}{2}\omega^{\SCH}_z & 0 &
 \lambda^{-1} e^{\frac{1}{2}\omega^{\SCHt}}\\
- \lambda^{-1} Q^{\SCH} e^{-\omega^{\SCHt}} & 
 - \frac{1}{2}\omega^{\SCH}_z & 0\\
0 & \lambda^{-1}e^{\frac{1}{2}\omega^{\SCHt}}& 0
\end{pmatrix},   \quad 
V_{\SCH}^{\lambda} =
 \begin{pmatrix} 
- \frac{1}{2}\omega^{\SCH}_{\bar z}
 & \lambda \overline{ Q^{\SCH}} e^{-\omega^{\SCHt}}  & 0 \\
 0 & \frac{1}{2}\omega^{\SCH}_{\bar z} & \lambda e^{\frac{1}{2}\omega^{\SCHt}}\\
\lambda e^{\frac{1}{2}\omega^{\SCHt}}& 0 & 0
\end{pmatrix}. 
\end{align*}
 It is clear that $\alpha_{\SCH}^{\lambda}|_{\lambda=1}$ is 
 the Maurer-Cartan form of the coordinate frame $\mathcal F_{\SCH}$ 
 of $f^{\SCH}$. 
 Then by the discussion in the previous section, we have the following 
 theorem. 
\begin{Theorem}\label{thm:flatconnectionsCH}
 Let $f^{\SCH} : \D \to \CH$ be a minimal Lagrangian surface in $\CH$ 
 and  let $\alpha_{\SCH}^{\l}$ be the family of Maurer-Cartan 
 forms defined in  \eqref{eq:alphalambda-origCH}.
 Then $\d + \alpha_{\SCH}^{\lambda}$ gives a family of flat connections
 on $\D \times \SUI$.
 
 Conversely, given a family of connections  $\d 
 + \alpha_{\SCH}^{\lambda}$ 
 on $\D \times \SUI$, where $\alpha_{\SCH}^{\lambda}$ is as in \eqref{eq:alphalambda-origCH}, then  $\d + \alpha_{\SCH}^{\lambda}$ belongs 
 to an associated 
 famiy of minimal Lagrangian immersions into $\CH$ 
 if and only if the connection is flat for all 
 $\lambda \in S^1$.
\end{Theorem}

%%%%%%%%%%%%%%%%%%%%%%%%%%%%%%%%%%%%%
\section{Timelike minimal Lagrangian surfaces in $\CHI$}\label{sc:mLiCHI}
 In this section, we discuss a loop group formulation of 
 timelike minimal Lagrangian surfaces in the complex projective plane $\CHI$.
 The detailed discussion can be found in \cite{DK:timelike}.
 Here we use that the subscripts $u$ and $v$ 
 denote the derivatives with respect to 
 $u$ and $v$, respectively, that is,
\[
 f_u = \partial_u f = \frac{\partial f}{\partial u}, \quad
 f_v = \partial_v f = \frac{\partial f}{\partial v}.
\]

\subsection{Basic definitions}
 Let
\begin{equation}\label{eq:P}
 \r{P} = 
\begin{pmatrix}
 0 & 1 & 0 \\ 
 1 & 0 & 0 \\ 
 0 & 0 & -1 
\end{pmatrix},
\end{equation}
 and consider the three-dimensional complex Hermitian flat space $\C^3_2$ 
 with signature $(2,1)$.
\begin{equation*}
 \langle z, w\rangle = z^T \r{P} \,  \!\bar w = 
 z_1 \overline{w_2} + z_2 \overline{w_1} -z_3 \overline{w_3}.
\end{equation*}
 Let $H^5_3$ be the indefinite sphere (note again that the signature
 of $\C^3_2$ is $(2,1)$)
\[
 H^5_{3} = \left\{ w \in \C^3_2 \;|\;  \langle w, w\rangle = -1\right\}.
 \]
 Then the two-dimensional indefinite complex hyperbolic space $\CHI$ 
 is
 \begin{equation}
 \CHI = \{ \mathbb C^{\times} w \;|\; w \in \C^3_2, \langle w,w \rangle < 0\}
 \end{equation}
 Then there exists the Boothby-Wang type fibration \cite{BW, DF} 
 $\pi : H^5_3 \to \CHI$ given by $w \mapsto 
 %{\rm span}_{\C^{\times}} \{v\} = 
 \C^{\times} w$.
 The tangent space of $H^5_3$ at $p \in H^5_3$ is
\begin{equation*}
 T_p H^5_3  = \{ w \in \C^3_2 \;|\;\Re \langle w,p \rangle =0\}.
\end{equation*}
 Moreover, the space 
 $\mathcal{H}_p = \{ w \in T_p H^5_3\;|\;  \langle w,p \rangle =0\}$
 is a natural horizontal subspace.
 The form $\langle \> ,\> \rangle$ is an indefinite Hermitian inner product on  
 $\mathcal{H}_u$  with real and imaginary components
\[
 \langle \>,\> \rangle = g(\>,\>) + i \Omega(\>,\>).
\]
 Hence $g$ is indefinite and $\Omega$ is a symplectic form.
 Put 
\[
 \UIT = \{ A : \C^3_2 \to \C^3_2 \;|\; \mbox{$\C$-linear, satisfying 
 $\langle A w, A q\rangle =\langle w, q\rangle $}\},
\]
 and $\SUIT = \{ A \in \UIT \;| \; \det A =1\}$. We note $\UIT = S^1 \cdot 
 \SUIT$ and that these are connected real reductive Lie groups with their 
 centers consisting of multiples of the identity transformation.
 Since, $\SUI$ and $\SUIT$ are isomorphic groups, 
 so they are both connected.
 The  groups $\UIT$ and $\SUIT$ act naturally on $H^5_3$ and $\CHI$.
 The group $\UIT$ acts transitively on both spaces. 
 Moreover, this action is 
 equivariant relative to $\pi$ and holomorphic on $\CHI$. 
 Using the base point $e_3 = (0,0,1)^T$ it is easy to verify
 \[
  H^5_3 = \UIT/\UtwoT \times \{1\}, \quad 
  \CHI = \UIT/\UtwoT \times S^1.
 \]
\subsection{Horizontal lift and fundamental theorem}
 We now consider a timelike Lagrangian immersion $f^{\ST}$ 
 from a surface $M$ into $\CHI$. 
 Then it is known that on an open and contractible subset $\D$
 of $M$, there exists a special lift into $H^5_3$, that is, 
 $\f^{\ST} : \D \to H^5_3, \pi \circ \f^{\ST} = f^{\ST}|_{\D}$, and 
\begin{equation}\label{eq:horizontalCHI}
 \langle \d\f^{\ST}, \f^{\ST}\rangle=0
\end{equation}
 holds, see \cite{DK:timelike}. The lift $\f^{\ST}$ will be called a \textit{horizontal lift} of $f^{\ST}$.
 The induced metric of $f^{\ST}$ is represented, by using the horizontal lift $\f^{\ST}$ 
 as
 \[
  \d s^2 =  \Re \langle \d \f^{\ST}, \d \f^{\ST} \rangle.
 \]
 Since the induced metric is Lorentzian, 
 we can take locally  null coordinates $(u, v)$  on $\D \subset M$.
 Then the horizontality condition \eqref{eq:horizontalCHI} implies
 $\langle \f^{\ST}_u, \f^{\ST}\rangle
 = \langle \f^{\ST}_v, \f^{\ST}\rangle=0$, and taking the derivative with respect to 
 $v$ of the first term and $u$ of the second term, respectively,
 we infer:
 \begin{equation}\label{eq:nullfT}
 \Omega(\f^{\ST}_{u}, \f^{\ST}_{v} ) = \Im \langle  
 \f^{\ST}_{u}, \f^{\ST}_{v}\rangle = 0,
\end{equation}
 that is, $\f^{\ST}$ is Legendrian.
 Moreover, since  we have chosen $u$ and $v$ as as null coordinates for $f^{\ST}$, we have
 \begin{equation}\label{eq:nullT}
 {\langle \f^{\ST}_u, \f^{\ST}_{u}\rangle} =  {\langle \f^{\ST}_v, \f^{\ST}_v \rangle} 
  =  0
 \quad \mbox{and} \quad \Re {\langle \f^{\ST}_u, \f^{\ST}_v \rangle} \neq 0.
 \end{equation} 
 One can assume without loss of generality that 
 $\Re {\langle \f^{\ST}_u, \f^{\ST}_v \rangle} > 0 $ holds.
 Therefore there exists a real function $\omega^{\ST}: \D \to \R$ such that 
\begin{equation}\label{eq:metricT}
 \langle \f^{\ST}_u, \f^{\ST}_v \rangle  = e^{\omega^{\STt}} \quad \mbox{and}\quad 
 \d s^2 = 2 e^{\omega^{\STt}} \d u  \d v.
\end{equation}
 We now consider the coordinate frame
 \begin{equation}\label{eq:coordinateframeCHI}
 \mathcal F_{\ST} = ( e^{-\frac{1}{2}\omega^{\STt}}\f^{\ST}_u, 
 e^{-\frac{1}{2}\omega^{\STt}} \f^{\ST}_v, \f^{\ST}).
\end{equation}
 It is straightforward to see that $\mathcal F_{\ST}$ takes values in 
 $\UIT$, that is, 
 \[
  \mathcal F_{\ST}^T \,P_0 \overline{\mathcal F_{\ST}} P_0  = 
 \id
 \] 
 holds.
 For what follows it will be convenient to lift 
 the mean curvature vector of $f^{\ST}$ from
 $T_{f^{\ST}(u, v)} $ to $T_{\f^{\ST}(u, v)}  H^5_3$.
 It is easy to verify that  the vectors $ \f^{\ST}_u, \f^{\ST}_v, 
 i \f^{\ST}_u, i\f^{\ST}_v, i\f^{\ST}$ span $T_{\f^{\STt}(u, v)}  H^5_3$
 and project under $\d \pi$ to $f^{\ST}_u,f^{\ST}_v,if^{\ST}_u, 
 if^{\ST}_v, 0$ respectively.
 In this sense we identify  the mean curvature vector 
 $H = H_1 ie^{-\frac{1}{2}\omega^{\STt}} f^{\ST}_u +  H_2 ie^{-\frac{1}{2}\omega^{\STt}} f^{\ST}_v$ of $f^{\ST}$ with the vector
 $H = H_1 ie^{-\frac{1}{2}\omega^{\STt}} \f^{\ST}_u  +  H_2 ie^{-\frac{1}{2}\omega^{\STt}} \f^{\ST}_v$.

\begin{Lemma}
 The coordinate frame $\mathcal{F}_{\ST}$  of a timelike 
 Lagrangian immersion 
 into $\CHI$ is a smooth map $\mathcal{F}_{\ST}: \D \rightarrow \UIT$.
 In particular, $\det \mathcal{F}_{\ST}$ is a smooth map  
 from $ \D $ to $S^1$.
 The Maurer-Cartan form  
\begin{equation}\label{eq:alphaT}
 \alpha_{\ST} = \mathcal F_{\ST}^{-1} \d \mathcal F_{\ST} = \mathcal F_{\ST}^{-1} 
 ({\mathcal F}_{\ST})_u\d u
 +  \mathcal F_{\ST}^{-1} ({\mathcal F_{\ST}})_v \d v =  \mathcal U_{\ST}  \d u  
+  \mathcal V_{\ST} \d v 
\end{equation}
 can be computed as 
\begin{align}
 \mathcal U_{\ST} = \begin{pmatrix} 
 \frac{1}{2}\omega^{\ST}_u + \ell  & m & e^{\frac{1}{2}\omega^{\STt}}\\
 -Q^{\ST}e^{-\omega^{\STt}} &-\frac{1}{2}\omega^{\ST}_u + \ell & 0\\
 0 &  e^{\frac{1}{2}\omega^{\STt}}& 0
 \end{pmatrix}, \quad 
 \mathcal V_{\ST} =
 \begin{pmatrix} 
 - \frac{1}{2}\omega^{\ST}_{v}+ m 
 & - R^{\ST}  e^{-\omega^{\STt}}  & 0 \\
 \ell & \frac{1}{2}\omega^{\ST}_{v}+ m & e^{\frac{1}{2}\omega^{\STt}}\\
 e^{\frac{1}{2}\omega^{\STt}}& 0 & 0
\end{pmatrix}, 
 \label{eq:UVT}
\end{align}
 where  $\ell =  \langle H, \f^{\ST}_u \rangle$,
 $m =  \langle H, \f^{\ST}_{v} \rangle$,
 $H$ denotes the mean curvature vector, and 
 $Q^{\ST}$ and $R^{\ST}$ are 
 {\rm purely imaginary functions} defined by
\begin{equation}\label{eq:QandRT}
 Q^{\ST} = \langle \f^{\ST}_{uuu}, \f^{\ST}\rangle\quad \mbox{and} \quad 
 R^{\ST} = \langle \f^{\ST}_{vvv}, \f^{\ST}\rangle.
\end{equation}
 Here we have used $ \langle H, \f^{\ST}_{v} \rangle = 
- \langle \f^{\ST}_v , H \rangle$ and $\langle H, \f^{\ST}_{u} 
 \rangle = 
- \langle \f^{\ST}_{u},H \rangle$.
 Moreover, $\ell$ and $m$ take values in $i \R$.
\end{Lemma}

\begin{Corollary}
 \label{MCT}
 For a $1$-form $\alpha_{\ST}$ satisfying  \eqref{eq:alphaT} 
 and  \eqref{eq:UVT}, the following statements hold$:$
\begin{enumerate}
 \item The {\rm mean curvature 1-form} $\sigma_H^{\ST}
 =  \Omega(H, \d \f^{\ST}) =\ell \d u + m \d v $
  satisfies $i \sigma_H^{\ST} = \langle H, \d\f^{\ST} \rangle 
 = \frac{1}{2} {\rm trace} ( \alpha_{\ST})$.

\item The $1$-form $\alpha^{\ST}$ satisfies the Maurer-Cartan equations 
 if and only if
 \begin{align*}
& \omega^{\ST}_{u v} - \left(1 - \frac{1}{2}|H|^2 \right) e^{\omega^{\STt}} 
+  Q^{\ST} R^{\ST} e^{-2 \omega^{\ST}} =0, \\
& \d \sigma^{\ST}_H =0, \quad 
 Q^{\ST}_{v} e^{-2\omega^{\STt}} = 
- (\ell e^{-\omega^{\STt}})_u , \quad 
  R^{\ST}_{u} e^{-2\omega^{\STt}} = - (m e^{-\omega^{\STt}})_v
 \end{align*}
\end{enumerate}
\end{Corollary}

 \begin{Theorem}[Fundamental theorem for Lagrangian immersions into
 $\CHI$]\label{thm:fundCHI}
 Assume $f^{\ST}:\D \rightarrow  \CHI$ is a conformal 
 Lagrangian immersion and let $\f^{\ST}$ denote one of 
 its horizontal lifts and $\mathcal{F}_{\ST}$ 
 the corresponding coordinate frame \eqref{eq:coordinateframeCHI}.
 Then $\alpha_{\ST} = \mathcal{F}_{\ST}^{-1} \d \mathcal{F}_{\ST} 
 = \mathcal{U}_{\ST} \d u + \mathcal{V}_{\ST} \d v$ 
 with $\mathcal{U}_{\ST}$ and $\mathcal{V}_{\ST}$ 
 have the form \eqref{eq:UVT} 
 and their coefficients satisfy the equations stated 
 in Corollary {\rm \ref{MCT}}.
 
 Conversely, given a functions $\omega^{\ST}, H$ on $\D$ together with 
 a cubic differential $Q^{\ST}\d u^3 + R^{\ST}\d v^3$ and 
 a $1$-form $\sigma^{\SCH}_H = 
 \ell \d u  + m \d v$ on $\D$ such that the conditions of 
 Corollary {\rm \ref{MCT}}
 are satisfied $($with $\langle H, \f^{\ST}_u \rangle$ 
 replaced by $m)$, 
 then there exists a solution $\mathcal{F}_{\ST} \in \UIT$ such that 
 $\f^{\ST} = \mathcal{F}_{\ST} e_3$ is a horizontal lift 
 of the null Lagrangian immersion $f^{\ST} = \pi \circ \f^{\ST}$.
\end{Theorem}
\subsection{Timelike minimal Lagrangian surfaces $\CHI$}
 If we restrict to minimal timelike Lagrangian surfaces, then the 
 equations \eqref{eq:UVT} together with $\ell= m = 0$ show that the determinant of the coordinate frame 
 is a constant (in $S^1$). So we can, and will, assume from here
 on that the horizontal lift of the given minimal immersion into 
 $\CHI$ is scaled (by a constant in $S^1$) 
 such that the corresponding
 coordinate frame $\mathcal F_{\ST}$ is in $\SUIT$. 
It is clear that the Maurer-Cartan form $ \alpha_{\ST} = \mathcal F_{\ST}^{-1} \d \mathcal F_{\ST} = \mathcal U_{\ST} \d u + \mathcal V_{\ST} \d  v$ of the minimal surface 
 is given by  
\begin{align}
 \mathcal U_{\ST} = \begin{pmatrix} 
 \frac{1}{2}\omega^{\ST}_u   & 0 & e^{\frac{1}{2}\omega^{\STt}}\\
 -Q^{\ST}e^{-\omega^{\STt}} &-\frac{1}{2}\omega^{\ST}_u & 0\\[0.1cm]
 0 &  e^{\frac{1}{2}\omega^{\STt}}& 0
 \end{pmatrix}, \quad 
 \mathcal V_{\ST} =
 \begin{pmatrix} 
 - \frac{1}{2}\omega^{\ST}_{v} 
 & - R^{\ST}  e^{-\omega^{\STt}}  & 0 \\
 0 & \frac{1}{2}\omega^{\ST}_{v} & e^{\frac{1}{2}\omega^{\STt}}\\
 e^{\frac{1}{2}\omega^{\STt}}& 0 & 0
\end{pmatrix}. 
 \label{eq:UVTmini}
\end{align} 
The integrability conditions stated in the corollary above 
 then are 
\begin{align}\label{eq:TzitzeicaCHT}
 \omega^{\ST}_{uv} - e^{\omega^{\STt}} +Q^{\ST}
 R^{\ST}e^{- 2 \omega^{\STt}}  =0,\quad \quad  Q^{\ST}_{v} = R^{\ST}_u =0.
\end{align}
 The first equation \eqref{eq:TzitzeicaCHT} is again one of the 
 \textit{Tzitz\'eica equations}.
 From the definition of $Q^{\ST}$ in \eqref{eq:QandRT}, it is clear that 
 \[
  C^{\ST}(u, v) = Q^{\ST}(u) \d u^3 + R^{\ST}(v) \d v^3 
 \]
 is the \textit{purely imaginary} cubic differential of 
 the timelike minimal Lagrangian surface $f^{\ST}$.
 Conversely, let $C^{\ST}$ be a cubic differential and 
 let $\omega^{\ST}$ be a solution of 
 \eqref{eq:TzitzeicaCHT}. Then there exists a frame $\mathcal F_{\ST}$ 
 taking values 
 in $\UIT$ and a timelike minimal Lagrangian surface given by  
 $f^{\ST} = \pi \circ (\mathcal F_{\ST} e_3)$, where $e_3 =(0, 0, 1)^T$.

\begin{Remark}
 The fundamental theorem in Theorem \ref{thm:fundCHI} is 
 still true for a timelike minimal Lagrangian immersions into $\CHI$.
 \end{Remark}

\subsection{Associated families of minimal surfaces and 
 flat connections}
 From \eqref{eq:TzitzeicaCHT}, it is easy to see that there 
 exists a one-parameter 
 family of solutions of \eqref{eq:TzitzeicaCHT} 
 parametrized by $\l \in \R^{+} = \{ \l \in \R\;|\; \l>0\}$; 
 The corresponding family $\{ \omega_{\ST}^{\l},  C_{\ST}^{\l}\}_{\l \in \R^+}$
 then satisfies
\[
  \omega_{\ST}^{\l} = \omega^{\ST}, \quad C_{\ST}^{\l} 
 = \l^{-3} Q^{\ST} \d u^3 + \l^{3} R^{\ST} \d v^3.
\]
 As a consequence, there exists a one-parameter family of timelike minimal 
 Lagrangian surfaces 
 $\{\hat f_{\ST}^{\l}\}_{\l \in \R^{+}}$ such that 
 $\hat f_{\ST}^{\l}|_{\l=1} = f^{\ST}$.
 The family $\{\hat f_{\ST}^{\l}\}_{\l \in \R^{+}}$ will be called the \textit{associated 
 family} of $f^{\ST}$.
 Let $\hat {\mathcal F}_{\ST}^{\l} $ be the coordinate frame of $\hat f_{\ST}^{\l}$.
 Then the Maurer-Cartan form $\hat \alpha_{\ST}^{\l}  = \hat {\mathcal U}_{\ST}^{\l} \d u + 
 \hat {\mathcal V}_{\ST}^{\l} \d v$ of $\hat {\mathcal F}_{\ST}^{\l}$
 for the associated family 
 $\{\hat f_{\ST}^{\l}\}_{\lambda \in \R^{+}}$
 is given by ${\mathcal U}_{\ST}$ and ${\mathcal V}_{\ST}$ as in 
 \eqref{eq:UVTmini} where we have replaced 
 $Q^{\ST}$ and $R^{\ST}$ by $\lambda^{-3} Q^{\ST}$ and $\lambda^3 R^{\ST}$, respectively.
 Then consider the gauge transformation $G^{\l}$ given by
 \begin{equation}\label{eq:extCHI}
 F_{\ST}^{\l}  =\hat {\mathcal F}_{\ST}^{\l} G^{\l}, \quad 
 G^{\l} = \di ( \lambda, \lambda^{-1}, 1).
 \end{equation}
 This implies
\begin{equation}\label{eq:alphaTlambda} 
\alpha_{\ST}^{\l} =  (F_{\ST}^{\l})^{-1}  \d F_{\ST}^{\l} = U_{\ST}^{\l} \d u + V_{\ST}^{\l} \d v	   
\end{equation}
 with $U_{\ST}^{\l} = (G^{\l})^{-1}\hat{\mathcal U}^{\l}_{\ST} G^{\l}$ 
 and $V_{\ST}^{\l} = (G^{\l})^{-1}\hat{\mathcal V}^{\l}_{\ST} G^{\l}$. 
 It is easy to see that $\hat{\mathcal F}^{\l}_{\ST} G^{\l} e_3= 
 \hat{\mathcal F}^{\l}_{\ST} e_3$. 
 Therefore  $f_{\ST}^{\l} = \pi \circ (\hat{\mathcal F}^{\l}_{\ST} G^{\l}e_3) =  \pi \circ (\hat{\mathcal F}^{\l}_{\ST} e_3)  
 = \hat{f}_{\ST}^\l$. 
 Hence we will not distinguish between 
 $\{\hat f_{\ST}^{\lambda}\}_{\lambda \in \R^{+}}$ and 
 $\{f_{\ST}^{\lambda}\}_{\lambda \in \R^{+}}$, and both families will be 
 called the associated family of $f^{\ST}$, and $F_{\ST}$ will also  be called the 
 coordinate frame of $f_{\ST}^{\l}$.
 
 %%%%%%%%%%%%%%%%%%%%%%%
 From the discussion in the previous section, 
 the family of Maurer-Cartan forms $\alpha_{\ST}^{\l}$  in \eqref{eq:alphaL}
 of a timelike minimal Lagrangian surface $f^{\ST}: M \to \CHI$  can be 
 computed explicitly as 
  \begin{equation}\label{eq:alphalambda-origT}
 \alpha_{\ST}^{\l} = U_{\ST}^{\lambda} \d u + V_{\ST}^{\lambda} \d v,
 \end{equation} 
 for $\lambda \in \C^{\times}$, where 
 $U_{\ST}^{\lambda}$ and $V_{\ST}^{\lambda}$ are  given by  
\begin{align*}
U_{\ST}^{\lambda} =
\begin{pmatrix}
 \frac{1}{2}\omega^{\ST}_u   & 0 & \l^{-1}e^{\frac{1}{2}\omega^{\STt}}\\
  - \l^{-1 }Q^{\ST}e^{-\omega^{\STt}} &- \frac{1}{2}\omega^{\ST}_u & 0\\
 0 &  \l^{-1}e^{\frac{1}{2}\omega^{\STt}}& 0
 \end{pmatrix}, \quad 
 V_{\ST}^{\lambda} =
 \begin{pmatrix} 
 - \frac{1}{2}\omega^{\ST}_{v}
 & - \l R^{\ST}  e^{-\omega^{\STt}}  & 0 \\
 0 & \frac{1}{2}\omega^{\ST}_v & \l e^{\frac{1}{2}\omega^{\ST}}\\
 \l e^{\frac{1}{2}\omega^{\ST}}& 0 & 0
\end{pmatrix}.
\end{align*}
 It is clear that $\alpha_{\ST}^{\lambda}|_{\lambda=1}$ is 
 the Maurer-Cartan form of the coordinate frame $\mathcal F_{\ST}$ of 
 $f^{\ST}$. 
 Then by the discussion in the previous section, we can characterize a minimal Lagrangian immersion in $\CHI$ in terms of 
 a family of flat connections.
\begin{Theorem}[\cite{DK:timelike}]\label{thm:flatconnectionsT}
 Let $f^{\ST} : \D \to \CHI$ be a timelike minimal Lagrangian surface in $\CHI$ and 
 let $\alpha_{\ST}^{\l}$ be the family of Maurer-Cartan forms defined in 
 \eqref{eq:alphalambda-origT}.
 Then $\d + \alpha_{\ST}^{\lambda}$ gives a family of flat connections
 on $\D \times \SUIT$.
 
 Conversely, given a family of connections  $\d 
 + \alpha_{\ST}^{\lambda}$ 
 on $\D \times \SUIT$, where $\alpha_{\ST}^{\lambda}$ is as 
 in \eqref{eq:alphalambda-origT}, then  $\d + \alpha_{\ST}^{\lambda}$ belongs 
 to an associated 
 famiy of minimal Lagrangian immersions into $\CHI$ 
 if and only if the connection is flat for all 
 $\lambda \in \R^+$.
\end{Theorem}

 %%%%%%%%%%%%%%%%%%%%%%%%%%
\section{Definite Proper Affine Spheres}\label{sc:affineD}
  In this section, we discuss a loop group formulation of definite proper affine spheres.
  The detailed discussion can be found in \cite{DW1, DW2}. 
 The general theory of 
 affine submanifolds can be found in \cite{NS}. We will use again complex coordinates and again restrict to surfaces defined on some simply-connected open subset $\D$ of $\C$. 
 %%%%%%%%%%%%%%%%%%%%%%%%%%
 
 \subsection{Basic definitions and results} \label{sc:B}
Classical affine differential geometry studies the properties of an immersed surface $f^{\SEA}: \D \rightarrow \R^3$ which are  invariant under the equi-affine transformations $f^{\SEA} \to Af^{\SEA}+b$, where $A \in \mathrm{SL}_{3}\R$ and $b \in \R^{3}$. The following form in local coordinates $(u_1,u_2)$ is naturally an equi-affine invariant:
\begin{equation}\label{eqfun}
\Lambda = \sum_{i,j} \det \left[ \frac{\partial^2 f}{\partial u_i \partial u_j}, \frac{\partial f}{\partial u_1}, \frac{\partial f}{\partial u_2} \right] (\d u_i \d u_j)\otimes (\d u_1 \wedge \d u_2) , 
\end{equation}
which induces an equi-affinely invariant quadratic form conformal to the Euclidean second fundamental form, called the \textit{affine metric} $g$, by $\Lambda = g \otimes \mathrm{vol}(g)$. Although the Euclidean angle is not invariant under affine transformations, there exists an invariant transversal vector field $\xi$ along $f(\D)$ defined by 
$\xi = \frac{1}{2} \Delta f $, called the \textit{affine normal}. Here 
 $\Delta$ is the Laplacian with respect to $g$. 

Another way to find the affine normal up to sign is by modifying the scale and direction of any transversal vector field (such as the Euclidean normal) to meet two natural characterizing conditions:
\begin{itemize}
\item[(i)] $D_X \xi^{\SEA} =\d \xi^{\SEA} (X) $ is tangent to the surface for any $ X \in T_p \D$,
\item[(ii)] $\xi^{\SEA}$ and $g$ induce the same volume measure on $\D$:
\[
 (\det \left[ f^{\SEA}_{\ast} X, f^{\SEA}_{\ast} Y, ~\xi^{\SEA} \right])^2
 =  |g(X, X)g(Y, Y)-g(X, Y)^2| 
\] 
for any $ X, Y \in T_p \D$. 
\end{itemize}
The formula of Gauss,
\begin{equation}\label{eqde}
    D_X f^{\SEA}_{\ast}Y = f^{\SEA}_{\ast}(\nabla_X Y) + g(X,Y)\xi^{\SEA} ~,
\end{equation}
 or the decomposition of $D_X f^{\SEA}_{\ast}Y$ into tangential 
 and transverse component,
 induces a torsion-free affine connection $\nabla$ on $\D$. Its difference with the  Levi-Civita connection $\nabla^g$ of $g$ is measured by the \textit{affine cubic form} defined  as: 
\begin{equation}\label{eqcu}
    C^{\SEA}(X,Y,Z):= g(\nabla_X Y-\nabla^g_X Y,Z).
\end{equation}
It is actually symmetric in all $3$ arguments. The \textit{affine shape operator} $S$ defined by the formula of Weingarten:
\[
 D_X \xi^{\SEA} = -f^{\SEA}_{\ast}(S (X)) ,
 \]
is self-adjoint with respect to $g$. The \textit{affine mean curvature} 
 $H$ and the \textit{affine Gauss curvature} $K$ are defined as 
 
\[
 H = 
 \frac{1}{2}\tr S \quad\mbox{and} \quad    K = \det S. 
\]

In the following we assume that the affine metric $g$ is definite. This means that  $f^{\SEA}(\D)$ is locally strongly convex and oriented (since its Euclidean second fundamental form is positive definite). Then there exist conformal coordinates $(x, y) \in \D$, that is, 
\[
g= 2 e^{\omega^{\SEAt}} (\d x^2 + \d y^2) = 2 e^{\omega^{\SEAt}}  |\d z|^2 = e^{\omega^{\SEAt}} (\d z \otimes \d \bar{z} + \d \bar{z} \otimes \d z), 
\]
 where   $z=x+iy$. 
 Then it is known that the affine normal $\xi^{\SEA}$ of a Blaschke immersion 
 can be represented in the form
\[
 \xi^{\SEA}=\frac{1}{2} \Delta f^{\SEA} = e^{-\omega^{\SEAt} } f^{\SEA}_{z \bar{z}}.
\]
 The affine normal $\xi^{\SEA}$ points to the concave side of the surface, and the orientation given by $i \, \d z \wedge  \d \bar{z}$ or $\d u \wedge  \d v$ is consistent with the orientation induced by $\xi^{\SEA}$. This $z$ coordinate essentially defines $\D$ as a uniquely determined Riemann surface. 

Alternatively we are studying \textit{affine-conformal} immersions $f$ of any Riemann surface $\D$ into $\R^3$:
\begin{equation}\label{eqconf}
\det [f^{\SEA}_{z} \; f^{\SEA}_{\bar{z}} \; f^{\SEA}_{zz}] = 0  
 = \det [f^{\SEA}_{z} \; f^{\SEA}_{\bar{z}} \;f^{\SEA}_{\bar{z}\bar{z}}],   \quad
 \mbox{and} \quad  
      \det[f^{\SEA}_{z} \; f^{\SEA}_{\bar{z}} \; f^{\SEA}_{z \bar{z} }]  = i e^{2\omega^{\SEAt} }.
\end{equation}
 The first condition here reflects that $f^{\SEA}$ is affine-conformal.
 Moreover, we introduce a function $Q^{\SEA}$ by  
\begin{equation} 
f^{\SEA}_{zz}  = \omega^{\SEA}_z f^{\SEA}_z + Q^{\SEA} e^{-\omega^{\SEAt}} f^{\SEA}_{\bar{z}}. 
\end{equation}
Then direct computations derive the fundamental affine invariants:  $g= 2e^{\omega^{\SEAt}}  | \d z|^2 $ by \eqref{eqfun} and $C^{\SEA}=Q^{\SEA} \d z^3+\overline{Q^{\SEA}} \d \bar{z}^3 $ by \eqref{eqde} and \eqref{eqcu}. We also have 
\begin{equation} \label{eqforQ}
\det[f^{\SEA}_{z} \; f^{\SEA}_{zz} \; f^{\SEA}_{zzz }] =  i (Q^{\SEA})^2 . 
\end{equation}
The shape operator $S$ has the special form 
\begin{equation} \label{shape}
S = \left(
      \begin{array}{cc}
        H & -e^{-2\omega^{\SEAt} } \overline{Q^{\SEA}}_z  \\
        -e^{-2\omega^{\SEAt} } Q^{\SEA}_{\bar{z}} & H  \\
      \end{array}
    \right) , 
\end{equation}
where $H=-e^{-\omega^{\SEAt} } \omega^{\SEA} _{z\bar{z}} - |Q^{\SEA}|^2 e^{-3\omega^{\SEAt} } $ is the affine mean curvature.

%%%%%%%%%%%%%%%%%%%%%%%
 \subsection{Maurer-Cartan form and Tzitz\'eica equation} \label{sc:MC}

The relations discussed above can also be illustrated by computing the evolution equations for the positively oriented frame 
 \[
  \mathcal{F}_{\SEA} = (e^{- \frac{1}{2}\omega^{\SEA}} f^{\SEA}_z, e^{- \frac{1}{2}\omega^{\SEAt}}f^{\SEA}_{\bar{z}}, e^{-\omega^{\SEAt} } f^{\SEA}_{z \bar{z}}),
 \]
 where we use $\xi^{\SEA} = e^{-\omega^{\SEAt} } f^{\SEA}_{z \bar{z}}$. Then 
$\det[f^{\SEA}_{z} \; f^{\SEA}_{\bar{z}} \; f^{\SEA}_{z \bar{z} }]   
 = i e^{2\omega^{\SEAt} }$ implies 
 $\det \mathcal{F}_{\SEA} = i$ and $ \mathcal{F}_{\SEA}(p_0)^{-1} \mathcal{F}_{\SEA}  \in \mathrm{SL}_3 \C$ follows for any base point $p_0 \in \D$.
 
\begin{Theorem}\label{DAMC}
 The Maurer-Cartan form 
\begin{equation} \label{eqfr}
 \mathcal F_{\SEA}^{-1} \d \mathcal F_{\SEA}= \mathcal F_{\SEA}^{-1} 
 ({\mathcal F}_{\SEA})_z \d z
 +  \mathcal F^{-1} ({\mathcal F}_{\SEA})_{\bar z} \d \bar z =  \mathcal U_{\SEA}  \d z  
+  \mathcal V_{\SEA} \d \bar z 
\end{equation}
 can be computed as 
\begin{align}  \label{DAMCUV}
 \mathcal U_{\SEA} = \begin{pmatrix} 
   \frac{1}{2}\omega^{\SEA}_z    & 0 & -H e^{ \frac{\omega^{\SEAt}}{2}} \\ 
  Q^{\SEA} e^{-\omega^{\SEAt}} & - \frac{1}{2}\omega^{\SEA}_z  & e^{-\frac{3}{2}\omega^{\SEAt}  } Q^{\SEA}_{\bar{z}}  \\
  0 & \! \! \! e^{ \frac{1}{2}\omega^{\SEAt}} & 0 \\
 \end{pmatrix}, \quad 
 \mathcal V_{\SEA} = 
 \begin{pmatrix} 
 -\frac{1}{2}\omega^{\SEA}_{\bar{z}} & \overline { Q^{\SEA}}  e^{-\omega^{\SEAt} }  & e^{-\frac{3}{2}\omega^{\SEAt}} \overline{Q^{\SEA}}_z \\
 0 & \frac{1}{2}\omega^{\SEA}_{\bar{z}} & -H e^{ \frac{1}{2}\omega^{\SEAt}} \\
  e^{ \frac{1}{2}\omega^{\SEAt}}  & 0 & 0
\end{pmatrix} . 
\end{align}
\end{Theorem}
The compatibility condition $(\mathcal{F}_{\SEA})_{z\bar{z}} = 
 (\mathcal{F}_{\SEA})_{\bar{z}z}$ (or the flatness of $\mathcal{F}_{\SEA}^{-1}  \d \mathcal{F}_{\SEA}$)  is equivalent to the  two structure equations: 
\begin{eqnarray}
\label{eq:FundaeqDA}
  H &=& - e^{-\omega^{\SEAt} } \omega^{\SEA}_{z\bar{z}} -  |Q^{\SEA}|^2 e^{-3\omega^{\SEAt} },  \\
 \label{eq:FundaeqDA2}
  H_{\bar{z}} &=&  e^{-3\omega^{\SEAt} } \overline{Q^{\SEA}} Q^{\SEA}_{\bar{z}} -  e^{-\omega^{\SEAt} } 
  (e^{-\omega^{\SEAt} }\overline{Q^{\SEA}}_z)_z .  
\end{eqnarray}
 The first equation is  the \textit{Gauss equation} and the second 
equation is the \textit{Codazzi equation} for $S$.
Altogether we have the following characterization of convex affine surfaces in $\R^3$.
\begin{Theorem}[Fundamental theorem for definite Blaschke
 immersions into $\R^3$]\label{thm:fundDA}
 Assume $f^{\SEA}: \D \rightarrow  \R^3$ is an affine-conformal immersion. Define $\omega^{\SEA}, Q^{\SEA}, H$ and the frame $\mathcal{F}_{\SEA}$ as above. Then its affine metric is $g= 2e^{\omega^{\SEAt}}  | \d z|^2 $,    its affine cubic form is $C^{\SEA}=Q^{\SEA} \d z^3+\overline{Q^{\SEA}} \d \bar{z}^3 $, and they satisfy the compatibility conditions \eqref{eq:FundaeqDA} and \eqref{eq:FundaeqDA2}, which are also equivalent to the flatness of $\alpha_{\SEA} = \mathcal{F}_{\SEA}^{-1} \d \mathcal{F}_{\SEA} = \mathcal{U}_{\SEA} \d z + \mathcal{V}_{\SEA} \d \bar{z}$ with $\mathcal{U}_{\SEA}$ and $\mathcal{V}_{\SEA}$ 
 having the form \eqref{DAMCUV}.
 
 Conversely, given a positive symmetric $2$-form $g=2 e^{\omega^{\SEAt}}  |\d z|^2$ and a symmetric $3$-form $C^{\SEA}=Q^{\SEA} \d z^3+\overline{Q^{\SEA}} \d \bar{z}^3 $ 
 on $\D \subset \C$ such that $H$ defined by  \eqref{eq:FundaeqDA} satisfies \eqref{eq:FundaeqDA2}, then there exists a surface $($unique up to affine motion$)$ such that $g,C^{\SEA}$ are the induced affine metric and affine cubic form respectively. 
\end{Theorem}

%%%%%%%%%%%%%%%%%%%%%%

\subsection{Definite affine spheres}
A \textit{definite affine sphere} is defined to be any affine surface with definite Blaschke metric having all affine normals meet at a common point which will be called its center, or where all affine normals are parallel. Equivalently an affine sphere is defined to be any ``umbilical'' affine surface (that is,
$S$ is a scalar function multiple of the identity everywhere).  

 By the matrix form \eqref{shape} of the shape operator $S$, 
 a definite affine sphere necessarily satisfies  $Q^{\SEA}_{\bar{z}} = 0 $, that is, $Q^{\SEA}$ is holomorphic. Then the above Codazzi equation \eqref{eq:FundaeqDA2} implies  $ H_{\bar{z}} = 0$, whence $H =$ const., since $H$ is real.
 
 %%%%%%%%%%%%%%%%%%%%
 \subsubsection{Types of affine spheres}
 So far we know that definite affine spheres have constant affine mean curvature $H$. Then a definite affine sphere is called \textit{elliptic, parabolic or hyperbolic}, when its affine mean curvature $H$ is positive, zero or negative respectively.
 
When $H=0$, it is also called ``improper''; and $\xi^{\SEA}$ is a constant vector which will usually be set to $(0,0,1)^t$ by some equi-affine transformation. Its center is at infinity. The only complete ones are paraboloids. 
 
When the shape operator $S$ in \eqref{shape} satisfies $S = H  \id \neq 0$, the corresponding affine sphere will be  called ``proper''. In this case we obtain 
 $\xi^{\SEA}=-H(f^{\SEA}-f^{\SEA}_0)$ with some $f^{\SEA}_0$ being the center of the affine sphere. For simplicity, we will always make $f^{\SEA}_0={\boldsymbol 0}$ 
 by translating the surface. 

\begin{Remark}
\mbox{}
\begin{enumerate}
\item Elliptic definite affine spheres have centers `inside' the surfaces and the only complete ones are ellipsoids. But the center of a hyperbolic definite affine sphere is `outside'.  They were considered in Calabi's conjecture for hyperbolic affine hyperspheres of any dimension (proved by Cheng-Yau \cite{CheYau86}, et al): 
Inside any regular convex cone $\mathcal{C}$, there is a unique properly embedded or complete 
 (with respect to the affine metric) hyperbolic affine sphere which has affine mean curvature $-1$, has the vertex  of $\mathcal{C}$ as its center, and is asymptotic to the boundary $\partial\mathcal{C}$. Conversely any properly embedded or complete hyperbolic affine sphere is asymptotic to the boundary of the cone $\mathcal{C}$ given by the convex hull of itself and its center.

\item 
It is 
clear that $Q  \d z^3$ is a globally defined holomorphic cubic differential (that is, in $H^0(M,K^3)$ where $K$ is the canonical bundle of $M$). Recall Pick's Theorem: $ C\equiv 0 $ if and only if $f(\D)$ is part of a quadric surface. So $Q$ is nonzero except for the quadrics. 

Near any point $z_0$ which is not any of the  isolated zeroes of $Q$ one could make a holomorphic coordinate change to normalize $Q$ to a nonzero constant, but we will not do that now, since then we have no control over the behaviour of $Q$ ``far away'' from $z_0$. The zeroes of $Q$ will be called ``planar'' points of the affine sphere.

\item We remark that the immersion is analytic for any definite affine sphere, since the defining equation is a fully nonlinear Monge-Ampere type elliptic PDE, see for example \cite[\S 76]{Blaschke}.
\end{enumerate}
\end{Remark}

It is easy to see that the Maurer-Cartan form 
\[
\alpha_{\SEA} = \mathcal F_{\SEA}^{-1} d \mathcal F_{\SEA} 
 =\mathcal{U}_{\SEA} \d z + \mathcal{V}_{\SEA} \d \bar{z}
\]
 of a definite affine sphere can be computed as
\begin{align}  \label{DASUV}
 \mathcal U_{\SEA} = \begin{pmatrix} 
   \frac{1}{2}\omega^{\SEA}_z    & 0 & -H e^{ \frac{1}{2}\omega^{\SEAt}} \\[0.1cm] 
  Q^{\SEA} e^{-\omega^{\SEA}} & - \frac{1}{2}\omega^{\SEA}_z  & 0  \\[0.1cm]
  0 & \! \! \! e^{ \frac{1}{2}\omega^{\SEAt}} & 0 
 \end{pmatrix}, \quad 
 \mathcal V_{\SEA} = 
 \begin{pmatrix} 
 -\frac{1}{2}\omega^{\SEA}_{\bar{z}} & \overline{ Q^{\SEA}}  e^{-\omega^{\SEAt} }  & 0 \\[0.1cm]
 0 & \frac{1}{2}\omega^{\SEA}_{\bar{z}} & -H e^{ \frac{1}{2}\omega^{\SEAt}} \\[0.1cm]
  e^{ \frac{1}{2}\omega^{\SEAt}}  & 0 & 0
\end{pmatrix} . 
\end{align}
In summary we obtain the governing equations for definite affine spheres in $\R^3$:
\begin{equation} \label{eqgc}
 \omega^{\SEA}_{z\bar{z}} + H e^{\omega^{\SEAt} } + |Q^{\SEA}|^2  e^{-2\omega^{\SEAt} } = 0, \quad Q^{\SEA}_{\bar{z}}  = 0. 
\end{equation}
 Moreover, given a holomorphic function $Q^{\SEA}$, 
 the first of the equations above is again a
 \textit{Tzitz\'eica equation}. 
\begin{Remark}
 The fundamental theorem in Theorem \ref{thm:fundDA} is 
 still true for a definite affine sphere into $\R^3$.
\end{Remark}

%%%%%%%%%%%%%%%%%%%%%%%

\subsection{A family of flat connections}
 From now on we will consider exclusively the case of proper definite affines spheres.  
 Then  we can and will scale the surface by a positive factor to normalize $H=\pm 1$.
 The following observation is crucial for the integrability 
 of definite affine spheres:
 The system \eqref{eqgc} is invariant under $Q^{\SEA} \rightarrow 
 \l^3 Q^{\SEA}$ for any $\l \in S^1$.
 Thus there 
 exists a one-parameter 
 family of solutions of \eqref{eqgc}
 parametrized by $\l \in S^1$; 
 The corresponding family 
 $\{\omega_{\SEA}^{\l},  C_{\SEA}^{\l}\}_{\l \in S^1}$
 then satisfies
\[
  \omega_{\SEA}^{\l} = \omega^{\SEA}, \quad C_{\SEA}^{\l} 
 = \l^{-3} Q^{\SEA} \d z^3 + \l^{3} \overline{Q^{\SEA}} \d \bar z^3.
\]
 As a consequence, there exists a one-parameter family of 
 definite affine spheres 
 $\{\hat f_{\SEA}^{\l}\}_{\l \in S^1}$ such that 
 $\hat f_{\SEA}^{\l}|_{\l=1} = f^{\SEA}$, which will be called the 
 \textit{associated family}.
 Let $\hat{\mathcal F}_{\SEA}^{\l}$ be the frame of $\hat f_{\SEA}^{\l}$.
 Then the Maurer-Cartan form $\hat{\alpha}_{\SEA}^{\lambda} = 
 (\hat{\mathcal F}_{\SEA}^{\l})^{-1} \d \hat{\mathcal F}_{\SEA}^{\l}
 = \hat{\mathcal U}_{\SEA}^{\l} 
 \d z +  \hat{\mathcal V}_{\SEA}^{\l} \d \bar z$
 can be computed as $\mathcal U_{\SEA}$ and $\mathcal V_{\SEA}$ in 
 \eqref{DASUV} replacing $Q^{\SEA}$ and $\overline{Q^{\SEA}}$ by 
 $\l^3 Q^{\SEA}$ and $\l^{-3}\overline{Q^{\SEA}}$, respectively.

 For the elliptic case (that is, $H=1$), applying the gauge 
 $G^{\l}=\di (i\lambda,i\lambda^{-1},1)$ to ${\hat \alpha}_{\SEA}^{\l}$, that is, 
\begin{equation}\label{eq:extEAP}
 F^{\l}_{\SEA + } := \hat {\mathcal F}_{\SEA}^{\l} G^{\l}_+
\end{equation}
 yields:
\begin{equation} \label{eqlae}
\alpha_{\SEA + }^{\lambda} =(F^{\l}_{\SEA + })^{-1} \d F^{\l}_{\SEA + } = 
 U_{\SEA + }^\l \d z + V_{\SEA + }^\l  \d\bar{z}
\end{equation}
where 
\begin{align} 
 U_{\SEA + }^\lambda  = \begin{pmatrix} 
    \frac{1}{2}\omega^{\SEA}_z   & 0 & i \lambda^{-1} e^{\frac{1}{2}\omega^{\SEAt}} \\ 
  \lambda^{-1} Q^{\SEA} e^{-\omega^{\SEAt} }& -\frac{1}{2}\omega^{\SEA}_z & 0  \\
  0 & i \lambda^{-1} e^{\frac{1}{2}\omega^{\SEAt} } & 0 \\
 \end{pmatrix}, \quad 
 V_{\SEA + }^\lambda =
 \begin{pmatrix} 
- \frac{1}{2} \omega^{\SEA}_{\bar{z}}  & \lambda \overline{ Q^{\SEA}}  e^{-\omega^{\SEAt} }  & 0 \\
 0 &  \frac{1}{2} \omega^{\SEA}_{\bar{z}}  &i \lambda e^{\frac{1}{2}\omega^{\SEAt}}  \\
 i \lambda e^{\omega^{\frac{1}{2}\SEAt}}  & 0 & 0
\end{pmatrix} . 
 \label{eq:UVe}
\end{align}

For the hyperbolic case (that is, $H=-1$), applying the gauge 
 $G_-^{\l} = \di (\lambda,\lambda^{-1},1)$ to $\hat \alpha_{\SEA}^{\l}$, that is, 
\begin{equation}\label{eq:extEAM}
 F^{\l}_{\SEA -} := \hat {\mathcal F}_{\SEA}^{\l} G^{\l}_-
\end{equation}
  yields: 
\begin{equation} \label{eqlah}
\alpha_{\SEA - }^{\lambda} =(F^{\l}_{\SEA - })^{-1} \d F^{\l}_{\SEA - } = 
U_{\SEA - }^\l \d z + V_{\SEA - }^\l  \d\bar{z}
\end{equation}
where 
\begin{align} 
 U_{\SEA - }^\lambda  = \begin{pmatrix} 
    \frac{1}{2}\omega^{\SEA}_z   & 0 & \lambda^{-1} e^{\frac{1}{2}\omega^{\SEAt}} \\ 
  \lambda^{-1} Q^{\SEA} e^{-\omega^{\SEAt} }& -\frac{1}{2}\omega^{\SEA}_z & 0  \\ 
  0 & \lambda^{-1} e^{\frac{1}{2}\omega^{\SEAt}} & 0 \\
 \end{pmatrix}, \quad 
 V_{\SEA -}^\lambda =
 \begin{pmatrix} 
- \frac{1}{2} \omega^{\SEA}_{\bar{z}}  & \lambda \overline{ Q^{\SEA}}  e^{-\omega^{\SEAt} }  & 0 \\
 0 &  \frac{1}{2} \omega^{\SEA}_{\bar{z}}  & \lambda e^{\frac{1}{2}\omega^{\SEAt} }  \\
  \lambda e^{\frac{1}{2}\omega^{\SEAt}}  & 0 & 0
\end{pmatrix}. 
 \label{eq:UVh}
\end{align}

%%%%%%%%%%%%%%%%%
%%%%%%%%%%%%%%%%%

 In both cases $\alpha_{\lambda}$ takes value in the  order $6$ twisted loop 
 algebra $\Lambda\sli_3 \C_\sigma$,  but it is contained in different real forms,
namely in the real forms induced by $\tau(X)=\ad(I_{2,1}P_0)\, \bar{X}$ for the hyperbolic case, and by 
 $\tau^\prime(X)=\ad (P_{0}) \, \bar{X} $ for the elliptic case. These two real forms are equivalent and both commute with $\sigma$, but, obviously, the associated geometries are very different.

\begin{Remark}
\mbox{}
\begin{enumerate}
\item 
Indeed definite affine spheres have two different geometries or elliptic PDE because there are two open cells in the corresponding Iwasawa decomposition, as explained in 
 \cite{DW1}: To simplify notation, denote this group of twisted loops $\Lambda\mathrm{SL}_3 \mathbb{C}_\sigma$ by $\cg$. Then $\cg_\tau$ and $\cg_+$ denote respectively the subgroups of $\tau$-real loops and the loops with holomorphic extension to the unit disc in $\C$. Iwasawa decomposition means the double coset decomposition $\cg_\tau \backslash \cg / \cg_+ $. The following observation makes it possible to have a unified treatment of elliptic and hyperbolic definite affine spheres.  Let $s_0:=\di (\lambda,-\lambda^{-1},-1)P_0$. There are exactly two open $\tau_2$-Iwasawa cells $\cg_{\tau_2} \cg_+$ and $\cg_{\tau_2} s_0 \cg_+$, which are essentially the same as two open $\tau_2^\prime$-Iwasawa cells (but interchanged):
\[
\cg_{\tau_2} s_0 \cg_+ = s_0 (\cg_{\tau_2^\prime} \cg_+), \qquad \cg_{\tau_2} \cg_+ = s_0 (\cg_{\tau_2^\prime} s_0 \cg_+) . 
\]

\item We may conjugate the complex frame to a real $\mathrm{SL}_3 \R$-frame:
\[
\mathcal{F}^{\R} := \ad \begin{pmatrix} \frac{1}{\sqrt{2}} & \frac{1}{\sqrt{2}} & 0 \\ \frac{\mathrm{i}}{\sqrt{2}} & \frac{-\mathrm{i}}{\sqrt{2}}& 0 \\ 0 & 0 & \sqrt{\mp 1} \end{pmatrix} \cdot  \mathcal{F}_\lambda.
\]
In fact $\mathcal{F}^{\R}=(e_{1},e_{2},\xi)$ with $\{e_{1},e_{2}\}$ being simply an orthonormal tangent frame with respect to  the affine metric. 
Recall that we obtain the immersion 
$f^{\SEA}=-\frac{1}{H} \xi^{\SEA}$ from the last column. It is clear now that we may also simply take the real part of the last column of $F_{\SEA \pm}^\lambda$ to get an equivalent affine sphere modulo affine motions.
\end{enumerate}
\end{Remark}

We now define the two subgroups $\widetilde \SLR^{\pm} \subset \SL$ 
\[
 \widetilde \SLR^{\pm} = \left\{ A \in \SL \;\Big|\;
\ad \begin{pmatrix} \frac{1}{\sqrt{2}} & \frac{1}{\sqrt{2}} & 0 \\ \frac{i}{\sqrt{2}} & \frac{-i}{\sqrt{2}}& 0 \\ 0 & 0 & \sqrt{\mp 1} \end{pmatrix}
\cdot  
 A \in \SLR
\right\}.
\]
It is easy to verify that both groups are isomorphic to  $\SLR$

It is remarkable that a simple  condition characterizes the extended frames of proper definite affine spheres:

\begin{Theorem}[\cite{DW1}] \label{thm:flatconnectionsDA}
 Let $f^{\SEA} : \D \to \R^3$ be a definite affine sphere
 in $\R^3$ and let $\alpha_{\SEA \pm}^{\l}$ be the family of Maurer-Cartan 
 forms defined in \eqref{eqlae} or \eqref{eqlah}.
 Then $\d + \alpha_{\SEA \pm}^{\lambda}$ gives a family of flat connections
 on $\D \times \widetilde \SLR^{\pm}$.

 Conversely, given a family of connections  $\d  
 + \alpha_{\SEA \pm}^{\lambda}$ 
 on $\D \times \widetilde \SLR^{\pm}$, where $\alpha_{\SEA \pm}^{\lambda}$ 
 is as in \eqref{eqlae} or \eqref{eqlah}, then  
 $\d + \alpha_{\SEA \pm}^{\lambda}$ belongs 
 to an associated 
 family of affine spheres into $\R^3$ 
 if and only if the connection is flat for all 
 $\lambda \in S^1$.
\end{Theorem}
\begin{proof} We have discussed the first part of the theorem above.
Concerning the converse direction we only show,
for simplicity, the hyperbolic (that is $H=-1$  in the flat connection \eqref{eqlah}). 
 The positive case is completely parallel.

 The reality conditions for $\sigma$ and $\tau_2$ guarantee that $\mathcal{F}^{-1}\mathcal{F}_{\bar{z}}$
 is affine  in $\lambda$. So we have
\begin{equation} \label{eqloop}
\mathcal{F}^{-1}\mathcal{F}_{z} = A \lambda^{-1} + B, \qquad \mathcal{F}^{-1}\mathcal{F}_{\bar{z}} = C \lambda + D,
\end{equation}
with $A \in \fg_{-1}$, $B \in \fg_0$, $C=\tau(A)$, and $D=\tau(B)$. The fixed points of both $\sigma$ and $\tau$ are of the form $\di (e^{\mathrm{i}\beta}, e^{-\mathrm{i}\beta}, 1 )$. Gauging by them respects the reality conditions. Let $e^{\mathrm{i}\beta}=\pm \frac{A_{13}}{|A_{13}|}$. Use it to scale  
 $A_{13}$ to a real positive function which 
 then is written in the form $e^{\psi/2}$.  The rest follows from the equations of flatness.
\end{proof}
\begin{Remark}
Recall the classical Tzitz\'eica equation for proper indefinite affine spheres (with no planar points):
\begin{equation} \label{eqtz}
\omega _{xy}=e^{\omega }-e^{-2\omega }
\end{equation}
We observe that the equation \eqref{eqgc} for hyperbolic definite affine spheres is the elliptic version of the above when $H=-1$ and $Q=1$. Both admit the trivial solution $\omega  \equiv 0$, and the corresponding surfaces are $x_1 x_2 x_3=1$ and $(x_1^2+x_2^2)x_3=1$ respectively. However, the equation \eqref{eqgc} for elliptic definite affine spheres admits no constant real solution, and some elliptic function examples will be given in 
 \cite{DW1}. 
\end{Remark}

%%%%%%%%%%%%%%%%%%%%%%%%%%%%
 
\section{Indefinite proper Affine spheres}\label{sc:affine}
 In this section, we discuss a loop group formulation of 
 indefinite proper affine spheres.
 The detailed discussion can be found in \cite{DE}.
\subsection{Blaschke immersions and its Maurer-Cartan}
 Let $f^{\SIA}: \D \to \R^3$ be a Blaschke immersion, that is, 
 there exists a unique \textit{affine normal} field $\xi^{\SIA}$
 (up to sign) such that the volume element of the affine metric 
 $\d s^2 = g$
 (which is determined by the second derivative of $f^{\SIA}$ and commonly called the 
 \textit{Blaschke metric})
 and the induced volume element on $\D \subset \R^2$ coincide, that is, 
 \begin{equation}\label{eq:volume}
  \det [f^{\SIA}_u, f^{\SIA}_v, \xi^{\SIA}]^2 = |g_{11} g_{22}-(g_{12})^2|
 \end{equation}
 holds.  
 In the following we assume that 
 the Blaschke metric $\d s^2 = g$ is indefinite.
 Then there exist null coordinates $(u, v) \in \D$ 
 \cite{TWeinstein} or 
 \cite[Prop 14. 1. 18]{B-Wood}, that is,
\begin{equation} \label{metrics}
 \d s^2 =  2 e^{\omega^{\SIA}} \d u \d v
\end{equation}
 holds for some real valued function $\omega^{\SIA} : \D \to \R$.
 Then the affine normal $\xi^{\SIA}$ can be represented as
\begin{equation}\label{eq:affienormal}
 \xi^{\SIA} = \frac{1}{2}\Delta f^{\SIA} = e^{- \omega^{\SIAt}}f^{\SIA}_{uv},
\end{equation}
 where $\Delta$ denotes Laplacian of the indefinite Blaschke metric.
 Combining \eqref{eq:volume} with \eqref{metrics}, 
 we have 
\[
\d s^2 = 2 \det [f^{\SIA}_u, f^{\SIA}_v, f^{\SIA}_{uv}] \;\d u \d v.
\]

 Note that the null coordinates can be rephrased as follows:
\begin{equation}\label{eq:affinemull}
\det [f^{\SIA}_{u} \;\;f^{\SIA}_{v} \;\; f^{\SIA}_{uu} ] \  = \  0 
 = \det [ f^{\SIA}_{u} \;\;f^{\SIA}_{v} \;\; f^{\SIA}_{vv}],  
 \qquad
 \det [f^{\SIA}_{u} \;\; f^{\SIA}_{v} \; \;f^{\SIA}_{u v } ] \  = \ 
e^{2 \omega^{\SIAt}},
\end{equation}
 see \eqref{metrics}.
 Moreover, we can introduce two functions 
\begin{equation}\label{eq:QandRIA}
 (Q^{\SIA})^2  = \det [f^{\SIA}_u, f^{\SIA}_{uu}, f_{uuu}^{\SIA}], 
 \quad \mbox{and} \quad - (R^{\SIA})^2 = \det [f^{\SIA}_v, f^{\SIA}_{vv},  f^{\SIA}_{vvv}].
\end{equation}
 From the definition of $Q^{\SIA}$ and $R^{\SIA}$ 
 in \eqref{eq:QandRIA}, it is clear that 
 \begin{equation} \label{cu2}
  C^{\SIA}(u, v)  = Q^{\SIA}(u, v) \d u^3 + R^{\SIA}(u, v) \d v^3
 \end{equation}
 is a cubic differential for the null Blaschke immersion $f^{\SIA}$.
 The shape operator $S=[s_{ij}]$, 
 which is defined by the Weingarten formula, 
 has relative to the basis $\{\partial_u, \partial_v\}$, where $u$ and $v$ are
 null coordinates, the special form: 
\[
 S = 
\begin{pmatrix} H & - e^{-2 \omega^{\SIAt}} Q^{\SIA}_v  \\ - e^{-2 \omega^{\SIAt}} R^{\SIA}_u &  H \end{pmatrix}.
\]
 Here $H \in \R$ is the affine mean curvature of $f^{\SIA}$.
 Then the coordinate frame of $f^{\SIA}$ is defined by
 \begin{equation} \label{coofr2}
\mathcal F_{\SIA}= (e^{-\frac{1}{2}\omega^{\SIAt}}f^{\SIA}_u, 
 e^{-\frac{1}{2}\omega^{\SIAt}} f^{\SIA}_v,  
 \xi^{\SIA} = e^{-\omega^{\SIA}} f_{uv}^{\SIA}),
 \end{equation}
 
and from \eqref{eq:volume}, it is easy to see that 
$\mathcal F_{\SIA}$ takes values in $\SLR$.
 Moreover, a straightforward computation shows that the following 
 lemma holds.
\begin{Lemma}
 The Maurer-Cartan form 
 \begin{equation} \label{eq:MCIA}
\alpha_{\SIA} =  \mathcal F_{\SIA}^{-1} \d \mathcal F_{\SIA} = \mathcal F_{\SIA}^{-1} 
 ({\mathcal F}_{\SIA})_u \d u  
 +  \mathcal F_{\SIA}^{-1} ({\mathcal F_{\SIA}})_v \d v 
 =  \mathcal U_{\SIA}  \d u  +  \mathcal V_{\SIA} \d v 
 \end{equation}
 can be computed as 
\begin{align}\label{eq:UVIA}
 \mathcal U_{\SIA} = \begin{pmatrix} 
 \frac{1}{2}\omega^{\SIA}_u   & 0 & -H e^{\frac{1}{2}\omega^{\SIAt}}\\
  Q^{\SIA}e^{-\omega^{\SIAt}} & -\frac{1}{2}\omega^{\SIA}_u & 
 e^{-\frac{3}{2} \omega^{\SIAt}} Q^{\SIA}_v \\
 0 & e^{\frac{1}{2}\omega^{\SIAt}}& 0
 \end{pmatrix}, \quad 
 \mathcal V_{\SIA} =
 \begin{pmatrix} 
 - \frac{1}{2}\omega^{\SIA}_v
 & R^{\SIA} e^{-\omega^{\SIAt}}  & e^{- \frac{3}{2} \omega^{\SIAt}} R^{\SIA}_u  \\
 0 & \frac{1}{2}\omega^{\SIA}_v & -H  e^{\frac{1}{2}\omega^{\SIAt}}\\
 e^{\frac{1}{2}\omega^{\SIAt}}& 0 & 0
\end{pmatrix}, 
\end{align}
\end{Lemma}

\begin{Corollary} \label{compSIA}
 The compatibility conditions for the system of equations stated just above are
\begin{gather}
\omega^{\SIA}_{uv} + H e^{\omega^{\SIAt}} +e^{- 2 \omega^{\SIAt}} 
 Q^{\SIA}R^{\SIA} =0, \\
e^{3 \omega^{\SIAt}} H_u = Q^{\SIA} R^{\SIA}_u -  
e^{2 \omega^{\SIAt}} (Q^{\SIA}_v e^{-  \omega^{\SIAt}} )_v, 
 \quad
e^{3 \omega^{\SIAt}} H_v = Q^{\SIA}_v R^{\SIA} -  
e^{2 \omega^{\SIAt}} (R^{\SIA}_u e^{-  \omega^{\SIAt}} )_u.
\end{gather}
\end{Corollary}
\begin{Theorem}[Fundamental theorem for indefinite Blaschke immersions]
\label{thm:BlaschkeIA}
Let $f^{\SIA}:\D \rightarrow \R^3$ be a Blaschke immersion with affine normal $\xi^{\SIA}$, 
indefinite Blaschke metric in null coordinates $u$ and $v$,
$ \d s^2 =2 e^{\omega^{\SIA}} \d u \d v$,  affine mean curvature $H$ 
 and cubic differential 
$C^{\SIA}  = Q^{\SIA} \;\d u^3 + R^{\SIA} \; \d v^3$. 
 Then the coordinate frame 
$\mathcal F_{\SIA}= (e^{-\frac{1}{2}\omega^{\SIAt}}f^{\SIA}_u, 
 e^{-\frac{1}{2}\omega^{\SIAt}} f^{\SIA}_v,  \xi^{\SIA} = e^{-\omega^{\SIA}} f_{uv}^{\SIA})$
 satisfies the Maurer-Cartan equation \eqref{eq:MCIA}. Here the coefficient matrices
 $\mathcal{U}_{\SIA}$ and $\mathcal{V}_{\SIA}$ 
 have the form \eqref{eq:UVIA}
 and their coefficients satisfy the equations stated 
 in Corollary {\rm \ref{compSIA}}.
 
 Conversely, given functions $\omega^{\SIA}, H$ on $\D$ together with 
 a cubic differential $Q^{\SIA}\d u^3 + R^{\SIA}\d v^3$ 
 such that the conditions of 
 Corollary {\rm \ref{compSIA}}
 are satisfied,
 then there exists a solution $\mathcal{F}_{\SIA} \in \SLR$ to the equation \eqref{eq:MCIA} such that 
 $f^{\SIA} = \mathcal{F}_{\SIA} e_3$ is an indefinite Blaschke immersion
 with null coordinates.
\end{Theorem}

%%%%%%%%%%%%%%%%%%%%%%
\subsection{Indefinite affine spheres}
From here on we will consider affine spheres.
As already pointed out in the last section this means that the shape operator $s$ is a multiple of the identity matrix.
We will also assume that the Blaschke metric is indefinite.
There are still two very different cases:

Case $H=0$:  these affine spheres are  called improper. They are very special and well known.
We will not consider this case.
Case $H \neq 0$:  such affine spheres are called \textit{proper}.
From now on,
 we will consider exclusively the proper case, and by 
 a scaling transformation
 we can assume that $H =-1$. Affine spheres with this property 
 are called 
 \textit{indefinite proper affine spheres}. 
 Then the Weingarten formula can be represented
 as 
\begin{align*}
\xi^{\SIA}_u= f^{\SIA}_u \quad \xi^{\SIA}_u = f^{\SIA}_u,
\end{align*}
 that is the affine normal $\xi^{\SIA}$ is the proper affine sphere 
 $f^{\SIA}$ itself
 up to a constant vector, that is, 
 $\xi^{\SIA} = f^{\SIA} + p$, where $p$ is some constant vector.
 By an affine transformation we can assume without loss of generality 
 $p =\boldsymbol 0$, and thus we have
 \[
  \xi^{\SIA} = f^{\SIA}.
 \]
 If we restrict to affine spheres, then the coefficient matrices  of the Maurer-Cartan equation
\begin{equation} \label{eq:AMCIA}
\alpha_{\SIA} =  \mathcal F_{\SIA}^{-1} \d \mathcal F_{\SIA} = \mathcal F_{\SIA}^{-1} 
 ({\mathcal F}_{\SIA})_u \d u  
 +  \mathcal F_{\SIA}^{-1} ({\mathcal F_{\SIA}})_v \d v 
 =  \mathcal U_{\SIA}  \d u  +  \mathcal V_{\SIA} \d v 
 \end{equation}
  are of the form
 \begin{align}\label{eq:AUVIA}
 \mathcal U_{\SIA} = \begin{pmatrix} 
 \frac{1}{2}\omega^{\SIA}_u   & 0 & e^{\frac{1}{2}\omega^{\SIAt}}\\
  Q^{\SIA}e^{-\omega^{\SIAt}} & -\frac{1}{2}\omega^{\SIA}_u & 
 0\\[0.1cm]
 0 & e^{\frac{1}{2}\omega^{\SIAt}}& 0
 \end{pmatrix},
\quad 
 \mathcal V_{\SIA} =
 \begin{pmatrix} 
 - \frac{1}{2}\omega^{\SIA}_v
 & R^{\SIA} e^{-\omega^{\SIAt}}  & 0  \\
 0 & \frac{1}{2}\omega^{\SIA}_v &  e^{\frac{1}{2}\omega^{\SIAt}}\\
 e^{\frac{1}{2}\omega^{\SIAt}}& 0 & 0
\end{pmatrix}. 
\end{align}
 Moreover, the integrability conditions now are
\begin{align}\label{eq:TzitzeicaInAff}
 \omega^{\SIA}_{uv} - e^{\omega^{\SIAt}} +e^{- 2 \omega^{\SIAt}} 
 Q^{\SIA} R^{\SIA} =0,\quad \quad  Q^{\SIA}_{v} = R^{\SIA}_u =0.
\end{align}
 The first equation in \eqref{eq:TzitzeicaCHT} is again a  \textit{Tzitz\'eica equation}.
 From the definition of $Q^{\SIA}$ and $R^{\SIA}$ in \eqref{eq:QandRIA}, it is clear that 
 \[
  C^{\SIA}(u, v) = Q^{\SIA}(u) \d u^3 + R^{\SIA}(v) \d v^3 
 \]
 is the \textit{real} cubic differential of the indefinite 
 affine sphere $f^{\SIA}$.  
 \begin{Remark}
 The fundamental theorem in Theorem \ref{thm:BlaschkeIA} is 
 still true for an indefinite affine spheres.
 \end{Remark}
%%%%%%%%%%%%%%%%%%%%%%%
\subsection{Associated families of indefinite affine spheres 
 and flat connections}
 From \eqref{eq:TzitzeicaInAff}, it is clear that there 
 exists a one-parameter 
 family of solutions parametrized by $\l \in \R^+,$ where 
 the original surface is reproduced for $\lambda = 1$.
 Then the corresponding family $\{ \omega_{\SIA}^{\l}, 
 C_{\SIA}^{\l}\}_{\l \in \R^+}$
 satisfies
\[
  \omega_{\SIA}^{\l} = \omega^{\SIA}, \quad C_{\SIA}^{\l} = \l^{-3} Q^{\SIA} \;\d u^3 + \l^{3} R^{\SIA}\; \d v^3.
\]
As a consequence, there exists a one-parameter family of 
 indefinite affine spheres $\{\hat f_{\SIA}^{\l}\}_{\l \in \R^+}$ such that 
 $\hat f_{\SIA}^{\l}|_{\l=1} = f^{\SIA}$.
 The family $\{\hat f_{\SIA}^{\l}\}_{\l \in \R^{+}}$ will be called the \textit{associated 
 family} of $f^{\SIA}$.
 Let $\hat {\mathcal F}_{\SIA}^{\l} $ be the coordinate frame of 
 $\hat f_{\SIA}^{\l}$.
 Then the Maurer-Cartan form $\hat \alpha_{\SIA}^{\l}  = 
 \hat {\mathcal U}_{\SIA}^{\l} \d u + 
 \hat {\mathcal V}_{\SIA}^{\l} \d v$ of $\hat {\mathcal F}_{\SIA}^{\l}$
 for the associated family 
 $\{\hat f_{\SIA}^{\l}\}_{\lambda \in \R^{+}}$
 is given by ${\mathcal U}_{\SIA}$ and ${\mathcal V}_{\SIA}$ as in 
 \eqref{eq:AUVIA} where we have replaced $Q^{\SIA}$ and 
 $R^{\SIA}$ by $\lambda^{-3} Q^{\SIA}$ 
 and $\lambda^3 R^{\SIA}$, respectively.

 Then consider the  gauge transformation $G^{\l}$ 
 \begin{equation}\label{eq:extIA}
 F_{\SIA}  =\hat {\mathcal F}_{\SIA}^{\l} G^{\l}, \quad 
 G^{\l} = \di ( \lambda, \lambda^{-1}, 1).
 \end{equation}
 This yields
\[
\alpha_{\SIA}^{\l} =  (F_{\SIA}^{\l})^{-1}  \d F_{\SIA}^{\l} 
 = U_{\SIA}^{\l} \d u + V_{\SIA}^{\l} \d v	   
\] 
 with $U_{\SIA}^{\l} = (G^{\l})^{-1}\hat{\mathcal U}_{\SIA}^{\l} G^{\l}$ 
 and $V_{\SIA}^{\l} = (G^{\l})^{-1}\hat{\mathcal V}_{\SIA}^{\l} G^{\l}$.
 It is easy to see
 that $\hat{\mathcal F}_{\SIA}^{\l} G^{\l} e_3 
 = \hat{\mathcal F}_{\SIA}^{\l} e_3$ holds. Define 
 $f_{\SIA}^{\l} = \hat{\mathcal F}_{\SIA}^{\l} G^{\l} e_3$. 
 Then we do not distinguish between 
 $\{\hat f_{\SIA}^{\lambda}\}_{\lambda \in \R^{+}}$ and 
 $\{f_{\SIA}^{\lambda}\}_{\lambda \in \R^{+}}$, and either 
 one will be 
 called the associated family, and $F_{\SIA}^{\l}$ will 
 also be called the 
 coordinate frame of $f_{\SIA}^{\l}$.
 
 %%%%%%%%%%%%%%%%%%%%%%%%%%%%%
 From the discussion in the previous section, 
 the family of Maurer-Cartan forms $\alpha^{\l}_{\SIA}$  
 of the indefinite proper affine sphere $f^{\SIA}:M \to \R^3$  can be 
 computed explicitly as 
  \begin{equation}\label{eq:alphalambda-origA}
 \alpha_{\SIA}^{\l} = U_{\SIA}^{\lambda} \d u + V_{\SIA}^{\lambda} \d v,
 \end{equation} 
 for $\lambda \in \C^{\times}$, where 
 $U_{\SIA}^{\lambda}$ and $V_{\SIA}^{\lambda}$ are  given by  
\begin{align*}
U_{\SIA}^{\lambda} = 
 \begin{pmatrix} 
 \frac{1}{2}\omega^{\SIA}_u   & 0 & \l^{-1}e^{\frac{1}{2}\omega^{\SIAt}}\\
 \l^{-1} Q^{\SIA}e^{-\omega^{\SIAt}} & -\frac{1}{2}\omega^{\SIA}_u & 0\\
 0 & \l^{-1}e^{\frac{1}{2}\omega^{\SIAt}}& 0
 \end{pmatrix}, \quad 
 V_{\SIA}^{\lambda} = 
 \begin{pmatrix} 
 - \frac{1}{2}\omega^{\SIA}_v
 & \l R^{\SIA} e^{-\omega^{\SIAt}}  & 0 \\
 0 & \frac{1}{2}\omega^{\SIA}_v & \l e^{\frac{1}{2}\omega^{\SIAt}}\\
 \l e^{\frac{1}{2}\omega^{\SIAt}}& 0 & 0
\end{pmatrix}.
\end{align*}
 It is clear that $\alpha_{\SIA}^{\lambda}|_{\lambda=1}$ is 
 the Maurer-Cartan form of the coordinate frame $\mathcal F_{\SIA}$ 
 of $f^{\SIAt}$. 
 Then by the discussion in the previous subsection, 
 we have the following theorem. 
\begin{Theorem}[\cite{DE}]\label{thm:flatconnectionsAf}
 Let $f^{\SIA}: \D \to \R^3$ be an indefinite proper affine sphere in $\R^3$ and 
 let $\alpha_{\SIA}^{\l}$ be the family of Maurer-Cartan forms defined in 
 \eqref{eq:alphalambda-origA}.
 Then $\d + \alpha_{\SIA}^{\lambda}$ gives a family of flat connections
 on $\D \times \SLR$. 

 Conversely, given a family of connections  $\d 
 + \alpha_{\SIA}^{\lambda}$ 
 on $\D \times \SLR$, where $\alpha_{\SIA}^{\lambda}$ is as 
 in \eqref{eq:alphalambda-origA}, then  $\d + \alpha_{\SIA}^{\lambda}$ 
 belongs  to an associated 
 family of indefinite affine spheres into $\R^3$ 
 if and only if the connection is flat for all 
 $\lambda \in \R^+$.
\end{Theorem}
%%%%%%%%%%%%%%%%%%%%%%%%%%%%%%%%%%%%
 
\section{Extended frames and the loop group method}\label{sc:dpw}

\subsection{Surfaces and extended frames}\label{subsc:surfext}
In the first five sections we started from five different general surface classes: Lagrangian immersions into $\CP$;
 Lagrangian immersions into $\CH$; Timelike Lagrangian immersions 
 into $\CHI$; Definite Blaschke surfaces in $\R^3$; 
Indefinite Blaschke surfaces in $\R^3$.

 For each of these surface classes we have introduced natural 
 frames (not always ``coordinate frames'' in the classical sense) 
 and have characterized them by  their ``shape''.
 The Maurer-Cartan equations of these frames were 
 (due to the special shape of the
 coefficient  matrices) integrable if and only if 
 a simple set of (highly non-trivial) equations was satisfied.

 Inside of each of the classes of surfaces listed 
 above we singled out a special type of surfaces.
 Respectively these were
\begin{itemize}
\item[$(\bullet_{\SCP})$] Minimal Lagrangian immersions into $\CP$,
\item[$(\bullet_{\SCH})$] Minimal Lagrangian immersions into $\CH$,
\item[$(\maltese_{\ST})$] Timelike minimal Lagrangian immersions into $\CHI$, 
\item[$(\bullet_{\SEA})$] Definite affine spheres in $\R^3$, 
\item[$(\maltese_{\SIA})$] Indefinite affine spheres in $\R^3$.
\end{itemize}
We showed that for all these special cases either a conformal parameter or a real 
(``asymptotic line'')  parameter is natural to choose for a ``convenient'' treatment.
 The cases with a preferable conformal parameter are indicated 
 by a $\bullet$ 
 and the other cases by a $\maltese$.
 Each of the classes of surfaces can be characterized 
 by a \textit{Tzitz\'eica equation}:
\begin{align*}
(\bullet_{\SCP})\quad&  \omega^{\SCP}_{z \bar{z}} +  e^{\omega^{\SCPt}} -
 |Q^{\SCP}|^2 e^{-2 \omega^{\SCPt}} =0, \quad 
 Q^{\SCP}_{\bar{z}} = 0,   \\
(\bullet_{\SCH}) \quad& \omega^{\SCH}_{z \bar{z}} 
 -  e^{\omega^{\SCHt}} + |Q^{\SCH}|^2 e^{-2 \omega^{\SCHt}} =0, \quad 
 Q^{\SCH}_{\bar{z}} = 0, \\
 (\maltese_{\ST}) \quad& \omega^{\ST}_{uv} - e^{\omega^{\STt}} +e^{- 2 \omega^{\STt}} Q^{\ST}
 R^{\ST} =0,\quad  Q^{\ST}_{v} = R^{\ST}_u =0, \\
(\bullet_{\SEA}) \quad& \omega^{\SEA}_{z\bar{z}} + H e^{\omega^{\SEAt} } + |Q^{\SEA}|^2  e^{-2\omega^{\SEAt} } = 0, \;\;(H = \pm 1), \quad Q^{\SEA}_{\bar{z}}  = 0,
  \\
 (\maltese_{\SIA}) \quad& \omega^{\SIA}_{uv} - e^{\omega^{\SIAt}} +e^{- 2 \omega^{\SIAt}} 
 Q^{\SIA} R^{\SIA} =0,\quad Q^{\SIA}_{v} = R^{\SIA}_u =0.
\end{align*}
 Note that $Q^{\ST}, R^{\ST}$ take values in $i \R$ and 
 $Q^{\SIA}, R^{\SIA}$ take values in $\R$, respectively.

For the conformal cases one can introduce a loop parameter
$\lambda \in S^1$ which produces an associated family of surfaces of the same type. 
For the asymptotic line  cases one can introduce a loop parameter
$\lambda \in \R_>0$ which produces an associated family of surfaces of the same type. 

 For general (non-geometric) purposes one can usually use $\lambda \in 
 \C^{\times}$.

The loop parameter was introduced in a special way:
Let  $\mathcal{F}$ denote the frame associated with a surface of one of the special classes listed above. Then we write
$\mathcal{F}^{-1} \d\mathcal{F} = \alpha$,
and write 
\[
 \alpha =\mathcal{F}^{-1} \d\mathcal{F} = \mathcal{U} \d a + 
 \mathcal{V} \d b,
\]
 where for the conformal case, $(a, b)$ is 
 given by  complex coordinates $(a, b) = (z, \bar z)$ 
 with $z = x + i y$, 
 and for the asymptotic line case, 
 $(a, b)$ is given by null coordinates, $(a, b) = (u, v)$
 with real $u, v$.
 Actually, one decomposes naturally in all cases $\mathcal{U}  = U_{-1} + U_0$
and $\mathcal{V}  = V_{1} + V_0$ and  introduces the 
 ``loop parameter''  $\lambda$ such that
\begin{equation}
 \alpha^\lambda = \lambda^{-1}U_{-1} \d a  + \alpha_0 + \l V_{1} \d b,
\end{equation}
 with $\alpha_0 = U_0 \, \d a  +V_0 \, \d b$.
 In fact $\alpha^{\l}$ is exactly a family of Maurer-Cartan forms $\alpha^{\l}_{*}$
 as in the previous five sections, where $*$ is one of $\CP$, $\CH$, $\CHI$, 
 $\mathbb A^3$ or $i \mathbb A^3$.
 The $1$-form $\alpha^\lambda$ will be called 
 \textit{the extended Maurer-Cartan form} and a unique
 solution to the equation
\begin{equation}\label{eq:extendedframe}
 ({F}^\lambda)^{-1} \d F^\lambda = \alpha^\lambda, \quad 
 F^{\l}(p_0) = \id
\end{equation}
 with some base point $p_0 \in \D$ will be called an \textit{extended frame}.
 Thus the coordinate frames $F^{\l}_{*}$ of the associated family of $f^{\l}_{*}$
 are in all five cases the extended frames up to an initial condition, 
 where $*$ is one of $\CP$, $\CH$, $\CHI$, 
 $\mathbb A^3$ or $i \mathbb A^3$.
 In all five cases we have stated a theorem saying 
\begin{Theorem}
 A surface is in the special class considered if and only if the
 family of Maurer-Cartan form $\alpha^\lambda$ yields 
 a flat connection $\d + \alpha^\lambda$.
\end{Theorem}
 Since in all our cases the special surface of actual interest 
 can be derived (quite) directly
 from the extended frame, 
 one of our goals is to construct all these frames.
\begin{Corollary}
The construction of all special surfaces listed above is equivalent to the construction 
of all the $1$-forms $\alpha^{\lambda} $.
\end{Corollary}

%%%%%%%%%%%%%%%%%%%%%%%%%%%%
\subsection{Flat connections and primitive frames}
To find all  $\alpha^{\lambda} $ (at least in an abstract sense) these 1-forms need to be described more specifically. To this end we consider the complex Lie algebra 
\begin{equation}
\mathfrak{g} = \sl
\end{equation}
and the order $6$ automorphism  $\hat \sigma$ of $\mathfrak{g}$ 
given by ($X \in \sl$):
\begin{equation} \label{defsigma}
\hat{\sigma} (X) =  - P X^T P, 
\end{equation}
where 
\begin{equation}
P = 
\begin{pmatrix}
0& \epsilon^2 & 0\\
\epsilon^4& 0 & 0\\
0&0&1\\
\end{pmatrix}
\left(
= \di (\epsilon^2, \epsilon^4, -1) P_0\right),
\end{equation}
 with $\epsilon = e^{\frac{ i\pi}{3}}$.
 Then on $\mathfrak{g}$ the automorphism $\hat{\sigma}$ has $6$ different eigenspaces 
\begin{equation}\label{eq:eigenspaces}
 \mathfrak{g}_j\subset \mathfrak g,  
\end{equation}
 such that $[\mathfrak{g}_i, \mathfrak{g}_j] \subset \mathfrak{g}_{i+j}
 \quad (\operatorname{mod} 6)$ holds  
 for the eigenvalues $\epsilon = e^{\frac{2 i\pi j}{6}}$ with 
 $j = 0, 1, 2, \dots, 5$. 
 Note that we then have for example  $\mathfrak{g}_{-1} =  \mathfrak{g}_5$ etc. 
 and we also have  $0 \subset \mathfrak{g}_0.$
 The crucial result for our discussion is:
\begin{Theorem} \label{propalpha}
 For all special surface classes the matrices $U_j$ and $V_j$
 are contained in the eigenspace of $\hat{\sigma}$ for the eigenvalue $e^{ 2 i \pi j/6}$, 
 that is, $U_j, V_j \in  \mathfrak{g}_j.$
More precisely we have 
\begin{equation} \label{primitive}
 \alpha^\lambda = \lambda^{-1}U_{-1} \d a + \alpha_0 + \lambda V_{1} \d b
 \in \mathfrak{g}_{-1} \oplus \mathfrak{g}_0 \oplus  \mathfrak{g}_1,
\end{equation}
where $a$ and $b$ denote the coordinates of the surface class under consideration.
 Moreover, for each special surface class there exists an anti-holomorphic involutory automorphism $\hat \tau$ of $\mathfrak{g}$ such that
 \begin{equation} \label{realform}
 \alpha^\lambda \in \mathfrak{g}^{\hat \tau}, 
 \end{equation}
 where $\mathfrak{g}^{\hat \tau}$ denotes the real subalgebra of $\mathfrak{g}$ consisting of all elements in $\mathfrak{g}$ which are fixed by 
 $\hat \tau$.
\end{Theorem}
\begin{Remark} In the conformal case we have the following statements:
\begin{enumerate}
 \item It is an important feature here that  $\hat \sigma$ maps $\alpha^\lambda$ to $\alpha^\mu$, 
$ \mu = \lambda e^{2 \pi i /6} \in \mathfrak{g}^{\hat \tau}$ .

\item The automorphism  $\hat \sigma$  leaves invariant $\mathfrak{g}^{
 \hat \tau}.$ 

\item The automorphisms $\hat \sigma$ and $\hat \tau$ commute on $\mathfrak{g}.$  
\end{enumerate}
 \end{Remark}
The situation in the asymptotic line case is quite different from what we just remarked.
 \begin{Theorem}
 Assume we have an immersion $f$ of split real type with extended frame 
 $F^\lambda$ and Maurer-Cartan form $\alpha^\lambda$. Let $\hat \tau$ be an involutory anti-holomorphic automorphism of $\mathfrak{g}$ 
 which fixes $\alpha^\lambda$.
 Writing 
 \[
  \alpha^\lambda = \lambda^{-1} U_{-1} \d u + (U_0\d u + V_{0} \d v) + \lambda V_{1} \d v,
 \]
 it follows that $\hat \tau$ fixes $U_{-1} + U_0$ and  $V_{0} + V_1$.
 Let us assume that $\hat \tau$ actually fixes all $U_j$ and all $V_j$.
 And let us assume also that the Lie algebra generated by 
 \[
  \{U_{-1}(u,v), \;U_0(u,v),\; V_0(u,v), \;V_1(u,v) \;|\; (u,v) \in \D\}
 \]
 generates the Lie algebra $\mathfrak{g}^{\hat \tau}$.
 Then $\hat \tau$ and $\hat \sigma$ satisfy the following relation$:$
 \begin{equation}
 \hat \sigma \hat \tau \hat \sigma = \hat \tau
 \end{equation}
 on $\mathfrak g$.
 \end{Theorem}
 \begin{proof}
 By our assumptions we obtain that $\hat \tau$ leaves each eigenspace of 
 $\hat \sigma$ in $\mathfrak{g}$  invariant.
 Hence $\sigma \circ \hat \tau \circ \hat \sigma(X_j) = \hat \sigma \circ \hat \tau (\epsilon^j X_j) =
 \hat  \sigma ( \epsilon^{-j}  \hat \tau(X_j)) =  \epsilon^{-j} \hat \sigma (\hat \tau(X_j)) = \hat \tau(X_j)$ for all eigenvectors  $X_j$ of $\hat \sigma.$
 \end{proof}
 More details will be explained in the following section of this paper.
 An extended frame $F^\lambda$ for which the Maurer-Cartan form 
 $\alpha^\lambda$ satisfies \eqref{primitive} and \eqref{realform}  
 will be called \textit{primitive} relative to $\hat \sigma$ and $\hat \tau$.

\begin{Corollary}
 In all our special surface classes  the extended frame is \it{primitive } 
 relative to $\hat \sigma$ and the real form $($anti-holomorphic$)$ automorphism 
 $\hat \tau$ chosen  for the  special surface class.
\end{Corollary}

%%%%%%%%%%%%%%%%%%%%%%%%%%%

\subsection{The loop group method for primitive extended frames}
 It is most convenient to explain the procedure for the conformal case 
 and for the asymptotic line case separately.

Let $\hat \sigma$ be as above and let $\hat \tau$ be the anti-holomorphic involutory automorphism
associated with the chosen surface class.
Let 

\[
 \mathfrak{g} = \sl, \quad  G =  \SL.
\]

By $G^{\hat \tau}$ and $\mathfrak{g}^{\hat \tau}$ we denote the corresponding fixed point group and algebra respectively. Actually, for $G^{\hat \tau}$ one could also use any Lie group between
$G^{\hat \tau}$ and its connected component.

 From what was said above, the extended frame  $F^\lambda$ of 
 an immersion of our special class is contained in 
 $G^{\hat \tau}$. The corresponding Maurer-Cartan form is contained in 
 $\mathfrak{g}^{\hat \tau}$.

 By the form of  $(F^\lambda)^{-1} \d F^\lambda$ we infer that 
 all the loop matrices  associated with geometric quantities 
 are actually defined for all $\lambda \in \C^{\times}$.
 In particular, all extended frames are defined on $S^1$.
 However, geometric interpretations are usually only possible for $\lambda \in S^1$ in the case of conformal case or $\l \in \R^+$ 
 in the case of asymptotic line case.

 Next one does no longer read the extended frame
 \[
  F^\lambda(a, b) = F (a, b,\lambda)
 \]
 as a family of frames, 
 parametrized by $\lambda \in S^1$, but as a function of $z$ into some loop group.
Here are the basic definitions:
\begin{enumerate}
\item The loop group of a Lie group $G$ is 
\[
       \Lambda G = \{ g: S^1 \rightarrow G \}.
\]
 Considering $G$ as a matrix group we use the Wiener norm on $S^1$ 
 and thus has a Banach Lie group structure on $\Lambda G $.
 Since all our geometric frames are defined for  $\lambda \in C^{\times}$, 
 we can apply the usual loop group techniques (see, for example 
 \cite[Theorem 4.2]{To}).

\item The \textit{plus} subgroup:
\[
       \Lambda^+G = \left\{g \in \Lambda G\;\Big|\;
 \begin{array}{l}
\mbox{$g$ as  a holomorphic extension to the open unit disk}   \\
 \mbox{and $g^{-1}$ has the same property.} 
 \end{array}
 \right\},
      \]
 and the normalized plus subgroup:
\[
 \Lambda_*^+G = \{ g \in  \Lambda^+G \;|\; g(0) = \id \}.
\]

\item The \textit{minus} subgroup:
 \[
  \Lambda^-G = \left\{ g \in \Lambda G\; \Big| \; 
 \begin{array}{l}
 \mbox{$g$ has  a holomorphic extension to the open  upper } \\
 \mbox{unit disk in $\C P^1$  and $g^{-1}$ has the same property.}
 \end{array}
\right\},
 \]
 and the normalized minus subgroup:
\[
 \Lambda_*^-G = \{ g \in  \Lambda^-G \; | \;g(\infty) = I \}.
\]
\end{enumerate}
 We now define automorphisms $\sigma$ and $\tau$ of $\Lambda G$
 as natural extensions of $\hat \sigma$ and $\hat \tau$ of $G$:
\begin{equation}\label{eq:sigmatau}
 \sigma (g)(\l) = \hat \sigma (g (\epsilon^{-1} \l)), \quad 
\tau (g) (\l)= \hat \tau ( g( B( \bar{\lambda})) ,
 \end{equation}
 where $B(\lambda) = \lambda ^{\pm 1}$ and $-1$ is 
 taken in the case of conformal type and 
 $+1$ is taken in the case of asymptotic line type.

\begin{enumerate}
\item[(4)]  The \textit{real} subgroup 
 \[
  \Lambda G^{\tau} = \{g \in \Lambda G\;|\; \mbox{$\tau (g)(\lambda) = g(\lambda)$.}\}.
 \]
\end{enumerate}
We will actually always use ``twisted subgroups'' of the groups above.
First we have
\[
 \Lambda G_{\sigma }= \{ g \in \Lambda G\;|\;  
 \mbox{$\sigma (g)(\lambda) = g(\lambda)$.} \}.
\]
The other twisted groups are defined analogously, like 
\[
 \Lambda_*^+G _{\sigma} = 
\Lambda_*^+G  \cap \Lambda G_{\sigma}.
\]
Finally, we actually use the twisted real loop group:
\[
 \Lambda G_{\sigma }^{\tau}= \{ g \in \Lambda G_{\sigma}\;|\;  
 \tau (g)(\lambda) = g(\lambda)  \}.
\]
\begin{Remark}
 The twisted real loop group may be defined as 
\begin{equation}\label{eq:Gtausigma}
 \Lambda G_{\sigma }^{\tau}=  \Lambda G_{\sigma } \cap 
 \Lambda G^{\tau},
\end{equation}
 if $\sigma$ and $\tau$ commute, these are  the cases of $(\bullet_*)$ in Section 
 \ref{subsc:surfext}, and if $\sigma$ and $\tau$ do not commute, 
 these are the cases of $(\maltese_*)$ in Section 
 \ref{subsc:surfext}, then 
 $\Lambda G_{\sigma }^{\tau}$
 cannot be defined as in \eqref{eq:Gtausigma}.
\end{Remark}

%%%%%%%%%%%%%%%%%%%%%%%%
\subsubsection{The loop group method for the conformal case}
 Let us fix a special surface class of conformal type.
 To understand the construction procedure mentioned above one considers next
 again an immersion of conformal type $f$ with primitive extended frame  
 $F$ relative to $\sigma$ and $\tau$ as above.

Then consider the linear ordinary differential equation in $\bar z$
\[
 \partial_{\bar z} L_{+}(z, \bar z, \l) = L_{+}(z, \bar z, \l) 
 \left(V_0(z, \bar z) + \l V_1(z, \bar z)\right), \quad 
  L_+(z_*, \bar z_*, \l) = \id.
\]
 Here we use the $\d \bar z$-coefficients in 
 $F^{-1} \d F = \alpha = \l^{-1}U_{-1} \d z + U_0 \d z + V_0 \d 
 \bar z + \l V_1 \d \bar z $ and consider $z$ and $\l$ as parameters of the 
 differential equation. Note that $U_0(z, \bar z) + \l V_1(z, \bar z)$
 takes values in the Lie algebra of $\Lambda^+ G_{\sigma}$, 
 thus $L_+(z, \bar z, \l)$ takes values in $\Lambda^+ G_{\sigma}$. On the 
 one hand, the primitive extended frame $F$ is also a solution of the
 above differential  equation, thus these two solutions should 
 coincide up to 
 an initial condition, that is, there exists $C(z, \l)$ which is 
 holomorphic in $z \in \D$ and $\l \in \C^{\times}$ such that 
 \begin{equation}\label{holframe}
 F(z, \bar z, \l) = C(z, \l) L_+(z, \bar z, \l)
 \end{equation}
 holds.
 
 Such a decomposition is always possible, 
 since $S^2$ does not occur in this paper as domain of a harmonic map.
 and defines a \textit{holomorphic potential} $\eta$ for $f$ by the formula
\[
 \eta = C^{-1} \d C.
\]
 The potential $\eta$ takes the form
\begin{equation} \label{eta}
\eta = \lambda^{-1} \eta_{-1}(z) \d z + \eta_{0}(z) \d z +\lambda^{1} \eta_{1}(z) \d z
+\lambda^{2} \eta_{2}(z) \d z + \cdots 
\end{equation}
We would like to emphasize:
\begin{enumerate}
\item All coefficient functions $\eta_j(z)$ are holomorphic on $\D$.

\item All  $\eta_j$ are contained in $\mathfrak{g}_j$, where 
 $\mathfrak g_j$ is defined in \eqref{eq:eigenspaces}.
\end{enumerate}
This explains the procedure to obtain a holomorphic potential from a primitive harmonic map.
The fortunate point is that this procedure can be reversed.

\begin{Theorem}[The loop group procedure for surfaces of conformal type]
Let $G, {\hat \sigma}$ and $\hat \tau $ as above.
Let  $f$ be an immersion of conformal type, $F(z, \bar z, \l) = 
 F^\lambda (z, \bar z)$
a primitive extended frame relative to $\hat \sigma$ and $\hat \tau$.
 Define $C$ by $F(z, \bar z, \lambda) = C(z,\lambda) \cdot L_+(z, \bar z, \lambda)$ 
 and put $\eta = C^{-1} \d C$, called a 
 {\rm holomorphic potential} for $f$.
 Then $\eta$ has the form stated in \eqref{eta}, the coefficient 
 functions $\eta_j$ of $\eta$ are holomorphic on $\D$ and we have 
 $\eta_j \in {\mathfrak{g}}_{j}$.

 Conversely, consider any holomorphic $1$-form $\eta$ satisfying the three conditions just listed for $\eta$.
 Then solve the ODE $\d C = C \eta$ on $\D$ with $C \in \Lambda G_{\sigma}$.
 Next write $C = F \cdot V_+$
 with $F \in \Lambda G^\tau_{\sigma}$ and $V_+ \in \Lambda ^+G_{{\sigma}} $.
 Then $F^{\l}(z, \bar z) = F(z, \bar z, \l)$ is the primitive extended frame 
 of some immersion $f$ of the class of surfaces under consideration.
\end{Theorem}

\begin{Remark}
\mbox{}
\begin{enumerate}
 \item  In the the procedure from $f$ to $\eta$ the decomposition 
$F(z, \bar z, \lambda) = C(z,\lambda) \cdot L_+(z,\bar z, \lambda)$
is always possible.
 In the converse procedure the decomposition (usually called 
 ``Iwasawa decomposition'', (see \cite{PS, Kell})
 is not always possible. But the set of points, where such a 
 decomposition is not possible is discrete in $\D$.

\item In the conformal case all geometric quantities like frame, potential etc. are actually real analytic on $\D$ and holomorphic in $\lambda \in C^{\times}$.

 \item In the conformal case we can start from a real Lie algebra $\mathfrak{q}$, say the one generated by
the Maurer-Cartan form $\alpha(z, \bar{z}), z \in \D$ of the coordinate frame of some immersion of conformal type. This always includes an automorphism $\kappa$ of this Lie algebra. Then, by carrying out the loop group procedure, we naturally and unavoidably  need to use the complexified Lie algebra $\mathfrak{q}^\C$. When extending the automorphism 
$\kappa$ complex linear to  $\mathfrak{q}^\C$ and defining $\rho$ as the anti-holomorphic automorphism of  $\mathfrak{q}^\C$ which defines  $\mathfrak{q}$ inside  
$\mathfrak{q}^\C,$
then we naturally obtain that $\kappa$ and $\rho$ commute. 
Hence immersions of conformal type always have to do with a complex linear automorphism and an anti-holomorphic involutory automorphism which commute. (Also see the Remark after Theorem \ref{propalpha}.)
\end{enumerate}
\end{Remark}

%%%%%%%%%%%%%%%%%%%%%%%
\subsubsection{The loop group method for the asymptotic line case}
The loop group method for this case looks at the outset very different.
And indeed, there are remarkable differences.
Since the scalar second order equation is not elliptic, solutions of low degree of differentiability can occur.
In this paper we always use only  functions which are as often differentiable as is convenient.
Since the loop parameter is for geometric quantities real now, we do not need to use the complex Lie group $G$ nor $\Lambda G$ etc., but always $G$ replaced by $G^\tau$, the real Lie group which is defined by $\tau$ and which is characteristic for the frame.

The main difference in procedure occurs at equation (\ref{holframe}). Since the coordinates $u$ and $v$ are on an equal basis (opposite to $z$ and $\bar{z}$) we need to carry out the splitting twice
\begin{equation} \label{potentialdu}
F(u,v, \lambda) = C_1(u,\lambda) \cdot L_+(u,v, \lambda), \quad \quad 
F(u,v, \lambda) = C_2(v,\lambda) \cdot L_-(u,v, \lambda).
\end{equation}
 Note that $L_+(u,v, \lambda)$ can be found by solving the differential 
 equation 

 \[
\partial_v L_+(u,v, \lambda)= L_+(u,v, \lambda)
 \left(V_0(u,v) + \l V_1(u, v)\right), \quad 
 L_+(u_*,v_*, \l) = \id.
\] 
 Here we use the coefficients in 
 $F^{-1} \d F = \alpha = \l^{-1}U_{-1} \d u + U_0 \d u + V_0 \d 
 \bar v + \l V_1 \d \bar v $ and consider $u$ and $\l$ as parameters.
 Since $V_0(u,v) + \l V_1(u, v)$ is given and smooth in 
 $u$ and in $v$,  also  
 $L_+(u,v, \lambda)$ is smooth in 
 $u$ and in $v$. Moreover, $V_0 + \l V_1$ takes 
 values in the Lie algebra of $\Lambda^+ G_{\sigma}$, thus $L_+$ takes values in 
 $\Lambda^+ G_{\sigma}$.As a consequence, there exists $C_1(u,\lambda)$ only depends 
 on $u$ and is smooth in $u$ and holomorphic in $ \lambda \in \C^{\times}$ such 
 that first equation in \eqref{potentialdu} holds.

 The argument for the second equation is, mutatis mutandis, the same.
 It is also important to observe that the two equations imply:
 \begin{equation} \label{matchinguandv}
  C_1(u,\lambda)^{-1} C_2(v,\lambda) = L_+(u,v, \lambda) L_-(u,v, \lambda)^{-1}.
 \end{equation}
 From this discussion we obtain a pair of potentials, 
 \[
  \eta_1 =  C_1(u,\lambda)^{-1} \partial_u C_1(u,\lambda)\d u
 \quad \mbox{and} \quad 
 \eta_2 =   C_2(v,\lambda)^{-1} \partial_v  C_2(v,\lambda) \d v.
 \]
 Analogous to the conformal case we also need to know what form the potentials $\eta_1$  and $\eta_2$ take.
\begin{align} \label{eta1}
\eta_1 &= \lambda^{-1} \eta_{1,-1}(u) \d u + 
\lambda^{0} \eta_{1,0}(u) \d 
 u +\lambda^{1} \eta_{1, 1}(u) \d u
+\lambda^{2} \eta_{1,2}(u) \d u + \cdots,  \\
\eta_ 2&= \lambda \eta_{2,1}(v) \d v + \lambda^{0} \eta_{2,0}(v) \d v +\lambda^{-1} \eta_{2, -1}(v) \d v
+\lambda^{-2} \eta_{2,-2}(v) \d v + \cdots. 
\label{eta2}
\end{align}
We would like to emphasize:
\begin{enumerate}
\item All coefficient functions $\eta_{m,j} \;(j =1, 2)$ are smooth 
 on some interval $\D_j \subset \R$.
\item All  the coefficient functions $\eta_{m,j}$ are contained in 
${\mathfrak{g}}_j^{\hat \tau}$.
\end{enumerate}
Note that here ${\mathfrak{g}}_j^{\hat \tau}$ are defined as
  \[
  \mathfrak{g}_j^{\hat \tau} := \mathfrak{g}^{\hat \tau} \cap \mathfrak{g}_j,
 \]  
 where $\mathfrak{g}_j$ is the eigenspace defined in \eqref{eq:eigenspaces}.

 As in the conformal case, one can also reverse the procedure.
 So let us start from two potentials $\eta_1(u,\lambda)$ 
 and $\eta_2(v,\lambda)$ satisfying the three conditions listed above.

 Next solve the pair of ODEs 
 \[
  \eta_1 =  C_1(u,\lambda)^{-1} \partial_u  
 C_1(u,\lambda)\d u\quad \mbox{and} \quad 
 \eta_2 =  C_2(v,\lambda)^{-1} \partial_v  
 C_2(v,\lambda) \d v
 \]
 for $C_1(u,\lambda)$ and $C_2(v,\lambda)$ with initial 
 conditions $C_1 (u_*, \l) = C_2(v_*, \l) = \id$.

 Next let us solve the equation
\begin{equation} \label{matchinguandvcon}
 C_1(u,\lambda)^{-1} C_2(v,\lambda) =
 L_+(u,v, \lambda)L_-(u,v, \lambda)^{-1}.
 \end{equation}
 Since $L_+(u,v, \lambda)$ and $L_-(u,v, \lambda)$ are in
 $\Lambda^+ G^\tau_\sigma$ and $\Lambda^- G^\tau_\sigma$ 
 respectively, equation \eqref{matchinguandvcon} is a 
 ``Birkhoff decomposition'' for $\lambda \in S^1$, see \cite{PS, Kell}.
\begin{Remark} Since, in general, the Birkhoff decomposition can not be carried 
 out for any loop matrices, there will be points, maybe curves, 
 where the  $L_{\pm} (u,v, \lambda) $ are singular.
\end{Remark}
 But away from singularities (\ref{matchinguandvcon}) 
 implies that there exists a matrix function $W(u,v,\lambda)$ satisfying
\begin{equation}
W(u,v,\lambda) =  C_1(u,\lambda) L_+(u,v, \lambda)   = 
C_2(v,\lambda) L_-(u,v, \lambda) .
\end{equation}

\begin{Theorem}[The loop group procedure for surfaces of asymptotic line  type]
 Let $G, \hat \sigma$ and $\hat \tau $ as above.
 Let  $f$ be an immersion of asymptotic line type, $F(u, v, \l) = 
 F^\lambda (u, v)$
 a primitive extended frame relative to $\hat \sigma$ and $\hat \tau$.
 Define $C_1$ and $C_2$ by $F(u, v, \lambda) = C_1(u,\lambda) 
 \cdot L_+(u, v, \lambda)$  and $F(u, v, \lambda) = C_2(v,\lambda) 
 \cdot L_-(u, v, \lambda)$ and put $\eta_i = C_i^{-1} \d C_i
 \; (i = 1,2)$, called a 
 {\rm pair of potential} for $f$.
 Then $\eta_i$ has the form stated in \eqref{eta1} and \eqref{eta2}, 
 the coefficient 
 functions $\eta_{i, j}$ of $\eta_i$ depends only on one variable 
 and we have  $\eta_{i, j} \in {\mathfrak{g}}_{j}^{\hat \tau}$.

 Conversely, consider any pair of $1$-forms $(\eta_1, \eta_2)$ satisfying the three conditions just listed for $\eta_i \; (i=1, 2)$.
 Then solve the ODEs $\d C_i = C_i \eta_i$ on $\D_i \subset \R$ with 
 $C_i \in \Lambda G_{\sigma}^{\tau}$.
 Next write $C_1^{-1} C_2   = L_+ L_-$ with $W = C_1 L_+ = C_2 L_-$
 with $L_{\pm } \in \Lambda ^{\pm}G_{{\sigma}}^{\tau}$. 
 Then there exist a gauge $F_0 \in G_0^{\hat \tau}$\footnote{
 $G_0^{\hat \tau} = \{ g \in G\;|\; \hat \sigma (g) = g\;\; \mbox{and}\;\; g \in G^{\hat \tau}\}.$} such that 
 $F^{\l}(u, v) = F(u, v, \l) F_0$ takes values in 
 $\Lambda G_{\sigma}^{\tau}$ is the primitive extended frame 
 of some immersion $f$ of the class of surfaces under consideration.
\end{Theorem}

\section{Complexification and real forms}\label{sc:real}
 This section is  a brief digression which is intended 
 to help to put this survey into a larger context. 
 It is clear that the extended frames $F$
 introduced in the previous sections take values 
 in the loop groups of 
\[
 \SU, \SUI, \SUIT,  \widetilde \SLR^{\pm} \; \mbox{or}\; \SLR.
\]
 For more details about these frames we refer to Section \ref{subsc:surfext}
 and the corresponding subsections of the first five sections.
 We show that their Maurer-Cartan forms  
 correspond to different  real forms of $\lsl$ or, more generally, of
 the affine Kac-Moody Lie algebra of type $A_2^{(2)}$.
 Moreover, by using the classification of real forms of type $A_2^{(2)}$ 
 in \cite{HG}, we obtain a rough classification of all surface classes associated 
 with specific real forms of  $\lsl$.

%%%%%%%%%%%%%%%%%% 

\subsection{Real forms of $\lsl$ and the surface classes considered in this paper}

 In the following discussion the Maurer-Cartan form $\alpha^{\l}$
 denotes $\alpha_{\SCP}^{\l}$, $\alpha_{\SCH}^{\l}$,  
 $\alpha_{\ST}^{\l}$, $\alpha_{\SEA+}^{\l}$,    $\alpha_{\SEA-}^{\l}$,    and $\alpha_{\SIA}^{\l}$
 in \eqref{eq:alphaL}, \eqref{eq:alphalambda-origCH}, \eqref{eq:alphaT},
 \eqref{eqlae}, \eqref{eqlah} and \eqref{eq:alphalambda-origA},
 respectively.
 Accordingly, the extended frame 
 $F^{\l}$ denotes $F_{\SCP}^{\l}$, $F_{\SCH}^{\l}$,  
 $F_{\ST}^{\l}$, $F_{\SEA+}^{\l}$ , $F_{\SEA-}^{\l}$ , and $F_{\SIA}$  in \eqref{eq:extCP}, \eqref{eq:extCH},
 \eqref{eq:extCHI},  \eqref{eq:extEAP},  \eqref{eq:extEAM} and 
 \eqref{eq:extIA}, respectively.
 A straightforward computation shows that 
 the Maurer-Cartan form $\alpha^{\lambda}$ of the  extended frame $F^{\lambda}$ 
 satisfies the following two 
 equations (where we write $\alpha (\lambda)$ for $ \alpha^\lambda$
 if it is convenient):
\[
 \sigma (\alpha)(\lambda) = \alpha (\lambda), \quad 
 \tau (\alpha)( \lambda) = \alpha(\lambda),
\]
 where  $\sigma$ is the order $6$
 linear outer automorphism of $\sl$ given by
 \[
\sigma (g)(\l) = - \ad (\di(\epsilon^2, \epsilon^4, -1)  \r{P} )\> g(\epsilon^{-1} \l)^T,
\]
with $\epsilon = e^{\pi i/3}$ the natural primitive sixth root of unity and 
\begin{equation}\label{eq:P0}
  \r{P} = \begin{pmatrix}
	   0 &1 &0  \\
	   1 &0 &0 \\
	   0 & 0 & -1
	  \end{pmatrix},
\end{equation}
 and $\tau$ is a complex  anti-linear involution 
 of $\sl$ varying with the surface class considered.
 
 Note, for simplicity we will sometimes write $\sigma(X) = - \ad(P)X^T$.

   More precisely, the family of Maurer-Cartan form $\alpha^{\l}$
 takes values in the following loop algebra:
\begin{equation}
 \lsl^{\tau} = \{ g : \C^{\times} \to 
 \sl\;|\; \sigma (g)(\lambda) = g(\lambda), \;\; 
\tau (g)(\lambda) = g(\l)\;\;\mbox{and $g \in \mathcal{W}$ }
 \},
 \end{equation}
 where $\mathcal{W}$ denotes the set of all $3\times3-$matrices with coefficients in the Wiener algebra on the unit circle which extend to all of $\C^{\times}$.
 
 Similarly, the extended frame $F(\l) = F^\lambda$ takes values in the 
 loop group $\LSL^{\tau}$ whose Lie algebra is
 $\Lambda \sl_{\sigma}^{\tau}$:
\begin{equation}
 \LSL^{\tau} = \{ g : \C^{\times} \to 
 \SL\;|\; \sigma (g)(\lambda) = g(\lambda), \;\; 
\tau (g)(\lambda) = g(\lambda)\;\;\mbox{and $g \in \mathcal{W}$}
 \},
\end{equation}
 where $\sigma$ is the order $6$ automorphism 
 \[
  \sigma (g) (\l)
 = \ad (\di(\epsilon^2, \epsilon^4, -1) \r{P} )\> g(\epsilon^{-1}\l)^{T -1},
 \]
and $\tau$ is, as above, an appropriate  complex  anti-linear involution.

 Note, by abuse of language we use the same notation 
 for  the Lie group automorphisms
 $\sigma$ and $\tau$ and their differentials. The order $6$ 
 automorphism $\sigma$ is in all cases the same.
 
 From the first five sections of this paper we obtain by inspection
 \begin{Theorem} \label{classify}
 The five surface classes discussed in the first five sections of this survey are related to complex anti-linear involutions $\tau$ as follows$:$ $\tau (g) (\l)$ is 
 given by 

\begin{tabular}{lcl}
 $(\bullet_{\SCP})$& $- \overline{g(1/\bar \lambda)}^T, $&  
  Minimal Lagrangian surfaces in $\CP,$ \cite{MM}$,$ \\
 $(\bullet_{\SCH})$ & $ - \ad (I_{2,1})  \overline
 {g(1/\bar \lambda)}^T,$ &
 Minimal Lagrangian surfaces in $\CH,$  \cite{ML}$,$  \\
$(\maltese_{\ST})$ & $ -\ad (\r{P})  \overline
 {g(\bar \lambda)}^T,$ & 
 Timelike minimal Lagrangian surfaces in $\CHI,$ \cite{DK:timelike}$,$ \\
$(\bullet_{\SEA})$ & $\ad (I_*\r{P})  \> \overline
 {g(1/\bar \lambda)}, $&
 Elliptic or hyperbolic affine spheres in $\mathbb {R}^3,$ \cite{DW1}$,$ \\
 $(\maltese_{\SIA})$ & $\overline
 {g(\bar \lambda)},$&
 Indefinite affine spheres in $\R^3,$ \cite{DE}$,$
\end{tabular}

 where $I_{2, 1} = \di(1, 1, -1)$ and $P_0$ is as just above. 
 Moreover, $I_*$ denotes $\id$ for the elliptic case and 
 $\id_{2,1}$ for the hyperbolic case.
\end{Theorem}
 The involutions $(\bullet_{\SCP}),(\bullet_{\SCH})$ and 
 $(\bullet_{\SEA})$  
 are called the \textit{almost compact} types and the remaining 
 ones $(\maltese_{\ST})$ and $(\maltese_{\SIA})$ are called  the \textit{almost split} types.
%%%%%%%%%%%%%%%%%%%
 \subsection{Real forms of $A_2^{(2)}$ and surface classes}
 Changing the point of view slightly we consider $\sigma$ as before and define the $\sigma$-twisted loop algebra
\[
  \lsl = \{ g : \C^{\times} \to 
 \sl\;|\; \sigma (g)(\lambda) = g(\lambda)
 \},
\]
 where we assume $g \in \mathcal{W}$, which denotes the set of 
 all $3\times3-$matrices with coefficients in the Wiener algebra 
 on the unit circle which extend to all of $\C^{\times}$.
  
 Similarly we consider the $\sigma$-twisted 
 loop group $\LSL$ whose Lie algebra is
 $\Lambda \sl_{\sigma}$:
\[
 \LSL = \{ g : \C^{\times} \to 
 \SL\;|\; \sigma (g)(\lambda) = g(\lambda)\},
\]
 Clearly, one can consider  $\lsl$ as the loop part of of 
 the twisted Kac-Moody algebra,  see for example \cite[Chapter 8]{Kac}:
\[
\hat{L}( \sl,  \sigma)  =  \lsl \oplus \C d \oplus \C c,
\]
 that is, $\hat{L}( \sl,  \sigma)$ is an extension of dimension $2$ 
 with  center $c$ of the loop 
  algebra $\lsl$. 
  Moreover, all the complex anti-linear involutions $\tau$ 
 considered above can be extended uniquely 
 to complex anti-linear involutions of the Kac-Moody algebra 
 $\hat{L}( \sl,  \sigma) $.
 This is a consequence of  in \cite[Theorem 3.4]{HG} as 
 $\lambda_\epsilon \in \pm 1$ in the notation of \cite{HG}.
 As a consequence of Theorem 3.8 in \cite{HG}, the equivalence 
 classes of involutions on the Kac-Moody algebra and the 
 loop algebra coincide. 
 
 From this point of view the complex anti-linear involutions 
 $\tau$ considered above 
 then define real forms of $\hat{L}( \sl,  \sigma) $.
  From \cite[Theorem 8.5]{Kac}, it follows that all 
 twisted Kac-Moody Lie algebras $\hat{L}( \sl,  \kappa)$, 
 with $\kappa$ an outer automorphism of $\sl$ are isomorphic. 
  
 Therefore, if we want to determine all possible real forms 
 (and the possible geometric counter parts) of all outer twisted 
 loop algebras  $\hat{L}( \sl,  \kappa)$, we can restrict to 
 $\kappa = \sigma$.
 So in our discussion below we can fix $\sigma$ and only need to vary
 the anti-linear involution $\tau$, the so-called {\it real form} 
 involution.
  Now we arrive at two different points of view:

  {\it Lie algebraic point of view}: 
 One classifies all real forms of the Kac-Moody algebra $A_2^{(2)}$ 
 up to conjugation. Any affine Kac-Moody algebra can be represented as the extension of a (possibly twisted) 
 loop algebra ${ \Lambda \mathfrak{g}}_{\sigma} =  
{\Lambda \sl} _{\sigma} = L(\sl,  \sigma).$ 
 While any suitable choice of $\mathfrak{g}$ and $\sigma$ uniquely defines an affine Kac-Moody algebra, the converse is not true: different involutions 
$\sigma$ and $\tilde{\sigma}$ may define the same Kac-Moody algebra, hence 
$\hat{L}(\sl,  \sigma)$ and $\hat{L}(\sl,  \tilde{\sigma})$ may be isomorphic for $\sigma \neq \tilde\sigma$. Hence, thinking about Kac-Moody algebras via pairs $(\mathfrak{g}, \sigma),$ the correct equivalence relation has to be slightly wider: it is defined in \cite{HG} and  called ``quasi-isomorphism''.
 Using the setting defined in loc. cit.,
 it turns out that the involutions listed in Theorem \ref{classify} are representatives (up to quasi-isomorphisms) of exactly 
 all real form involutions of $\hat{L}( \sl,  \sigma)$.
  Thus each representative of a real form of $\hat{L}( \sl,  \sigma)$ 
 has some geometric counter part.
 For all five geometric cases listed above 
 a loop group procedure has  been developed which allows (at least in principle) to construct all the surfaces of the corresponding class (see the references in Theorem \ref{classify}).
 This is a consequence of the fact that these surfaces can be characterized by a certain ``Gauss map'' to be harmonic.
 Actually, a harmonic Gauss map has only been established explicitly in 
 cases $(1)$ and  $(3)$ so far. In all other cases the existence of a harmonic Gauss map can be concluded, since the Maurer Cartan form of the naturally associated moving frame admits the insertion of a parameter $\lambda $ in such a way as it is known to correspond to a primitive harmonic map. 
 
 {\it Geometric point of view}: Here one wants to classify all 
 classes of surfaces which can be constructed as the five examples discussed in the first five sections of this paper, since the five $\tau$ listed in Theorem \ref{classify}
 all induce a surface class, the question is whether also quasi-isomorphic $\tau$ and $\tilde {\tau}$ can induce different surface classes.
 To determine all possible $\tau$ we recall that the known almost compact type surfaces 
 had $\tau's$ which commute with $\sigma,$ while the almost split type surfaces had $\tau's$ which satisfied the relation $\sigma \tau \sigma = \tau.$
 %%%%%%%%%%%%%%%%%%%%%%%%%%%%%%%
\subsection{Real form involutions}
It is known that all real form involutions $\tau$ of 
 $\lsl$ are induced from some complex anti-linear involution of 
 $\sl$, see \cite{HG}. Since we restrict for now our concentration  
 on $\sl$ it is fairly easy to reduce 
 the possibilities. 
\begin{Remark}
 It is known \cite{HG} that some real forms of 
 ``untwisted'' loop algebras such as $A_n^{(1)}$ are 
 not coming from \textit{any} real form involutions on underlining 
 finite dimensional Lie algebras. 
\end{Remark}

\subsubsection{Real form involutions commuting with $\sigma$}
 We now classify real form involutions commuting with $\sigma$.
 \begin{Proposition}
 Let $\tau$ be a real form involution of the loop algebra $
 \lsl = L( \sl,  \sigma)$ which commutes with $\sigma$. 
 We will use $\beta(X) = \bar{X}$ and $\tau_0 (X) 
 = - \bar{X}^T$. 
\begin{enumerate}
 \item[(a)] If $\tau = \ad (B) \circ \beta,$ then $B$ is a generalized permutation matrix coinciding with  $P_0$  after setting all non-zero coefficients equal to $1$.
 More precisely, after removing appropriate cubic roots and after possibly a conjugation by $\ad (D)$ with some diagonal matrix $D$ such that 
 $\ad (D)$ commutes with $\sigma$ we obtain $B = P_0$ or $B = I_{21}P_0$.
 
 \item[(b)] If $\tau$ is of the form $\tau = \psi \circ \beta$ with $\psi $ an outer automorphism of  $\sl$, then we write $\tau = \ad (Q) \circ \tau_0$. Then $Q$ is without loss of generality
  a diagonal matrix of the form $Q = \di 
 (q,q^{-1},1), q \in \R^{\times}$.
\end{enumerate}
 \end{Proposition}
 
 \begin{proof}
 In the following we denote 
 the restrictions of the $\sigma$ and $\tau$ on $\lsl$
 to the finite dimensional Lie algebra $\sl$
 by the same symbols.

 (a) Since $\tau$ commutes with $\sigma$, it also commutes with 
 $\sigma^2 = \ad (\Omega)$, where 
 $\Omega = \di (\epsilon^4, \epsilon^2,1).$  A direct evaluation yields 
 \begin{equation} \label{comm-case-a}
 B = \mu \Omega B \Omega.
 \end{equation}
 This is equivalent to $B_{ij} = \mu \Omega_{ii} \Omega_{jj} B_{ij}$.
 Clearly, the definition of $\Omega$ implies that $ \Omega_{ii} \Omega_{jj} $ only attains the values $\epsilon^4, \epsilon^2,1$. 
 It is straightforward to verify:
 \begin{align*}
 \Omega_{ii} \Omega_{jj}  &= 1 \Longleftrightarrow (i,j) \in \{(1,2), (2,1),(3,3)\}, \\
 \Omega_{ii} \Omega_{jj}  &= \epsilon^2  \Longleftrightarrow (i,j) \in \{(3,2), (2,3),(1,1)\}, \\
  \Omega_{ii} \Omega_{jj}  &= \epsilon^4  \Longleftrightarrow (i,j) \in \{(1,3), (3,1),(2,2)\}
\end{align*}
 Thus $B$ is a ``generalized permutation matrix''.
 
 Finally we need to evaluate the commutation relation with $\sigma$ directly. Writing this out  yields the equivalent equation
 \begin{equation} \label{commutesigma}
 P(B^{T})^{-1} = \rho B \bar{P}.
 \end{equation}
 Replacing all non-zero coefficients in this equation by $1$ still yields a correct equation.
 Since now the ``reduced equation'' reads 
 $\hat{P}(\hat{B}^{T})^{-1} = \hat{B} \hat{P}$, it follows 
 $\hat{B}  = \hat{P}.$ Hence $B$  has non-zero entries exactly, where $P$ has them.
 Evaluating (\ref{commutesigma}) explicitly yields four equations and one infers $B_{33}^3 = -1$.
 Hence $B_{33} = \epsilon,  \epsilon^3$ or $\epsilon^5$.
 For these cases one pulls out of $B$ the matrix $(-  \epsilon^m) I$ 
 and obtains without loss of generality $B_{33} = -1.$ Putting 
 $x = B_{12}$, then the (\ref{commutesigma})  also implies
 $B_{21} = x^{-1}$. 
 
 Evaluating the involution property of $\tau$ implies that $x$ is real.
 Now we put 
 \[
 D = \di\left( \sqrt{|x|}^{\frac{1}{2}},  \sqrt{|x|}^{-\frac{1}{2}}, 1
 \right)
 \]
 and consider
 $\hat{\tau} = \ad (D) \circ \tau \circ \ad (D)^{-1}$ and $\hat{\sigma} = \ad (D) \circ \sigma \circ \ad (D)^{-1}$. A straightforward computation yields $\hat{\sigma} = \sigma$ and 
 $\hat{B} = I_{2,1}P_0$ or  $\hat{B} = -P_0$. Clearly the minus sign is irrelevant and we obtain the claim.
  
 (b) By evaluating the first line in Theorem \ref{classify} we know that  $\tau_0$ commutes with $\sigma$. Hence the $\C$-linear automorphism $\ad (Q)$ commutes with $\sigma$, 
 whence it also commutes with $\sigma^2$ and therefore $Q$ is a diagonal matrix. A direct evaluation of the commutation property now yields $QP = \mu PQ^{-1}$. Taking the determinant 
 yields $\mu^3 = 1$ and the equation yields $\mu = Q_{33}^2$ and $\mu = Q_{11} Q_{22}$.
 Hence $Q_{33}^3 = 1$ and we can pull out without loss of generality
 $Q_{33} I$ from $Q$. Finally we evaluate the consequence of $\tau$ being an involution and obtain the claim.
 \end{proof}
 
 \begin{Corollary}
 The cases $(\bullet_{\SCP}),(\bullet_{\SCH})$ and $(\bullet_{\SEA})$ in 
 Theorem $\ref{classify}$, with case $(\bullet_{\SEA})$ split into two 
 cases, are exactly all possible geometric cases, where $\tau$ and $\sigma$ commute.
 \end{Corollary}

 %%%%%%%%%%%%%%%%%%%%%%%%%
 \subsubsection{Real form involutions satisfying $\sigma \tau \sigma = \tau$}
 In this case we proceed very similarly to the previous case.
\begin{Proposition}
 Let $\tau$ be a real form involution of the loop algebra 
 $\lsl = L( \sl,  \sigma) $ which 
 satisfies  the relation $\sigma \tau \sigma = \tau$. 
 As above we will use $\beta(X) = \bar{X}$ and $\tau_0 (X) 
 = - \bar{X}^T$.
\begin{enumerate}
\item If $\tau = \ad (B) \circ \beta,$ then $B $ is a diagonal matrix coinciding with $I$ after 
 removing appropriate cubic roots and after possibly a conjugation by 
 $\ad (D)$ with some diagonal matrix $D = \di 
 (\delta, \delta^{-1},1)$ such that $\ad (D)$ commutes with $\sigma$.
 
\item If $\tau$ is of the form $\tau = \psi \circ \beta$ with $\psi $ an outer automorphism of  $\sl$, then writing $\tau = \ad (Q) \circ \tau_0$ we obtain that $Q$ is, up to manipulations as in the proof of the last proposition, the matrix $P_0$. 
\end{enumerate}
 \end{Proposition}
 \begin{proof}
(a) Evaluating the defining equation one obtains
\begin{equation} \label{relsigma1}
P B^{T-1}P = \kappa B, 
\end{equation}
for some $\kappa$ satisfying $\kappa^3 = 1.$ 
Since we also have $\sigma^2 \tau \sigma^2 = \tau$, we also obtain
(recall: $\sigma^2(X) = 
\Omega X \Omega^{-1}$ with $\Omega = \di (\alpha^2, \alpha,1)$, $\alpha = \epsilon^2$).
\begin{equation} \label{relsigma2}
\Omega B \bar{\Omega} = \eta B
\end{equation}
with $\eta^3 = 1$.
Evaluating the last equation one observes that there are three cases: if one of the entries 
$B_{11}, B_{22},B_{33}$ is non-zero, then $B$ is a diagonal matrix.
If one of the entries $B_{12},B_{21},B_{31}$ does not vanish, then $\eta = \alpha$ and $B$ is a generalized permutation matrix associated with the permutation $(1,2,3) \rightarrow (3,1,2)$.
If one of the   three remaining entries of $B$ does not vanish, the $\eta = \alpha^2$ and $B$
 corresponds to the permutation  $(1,2,3) \rightarrow (2,3,1)$.
 
 Next we evaluate that $\tau$ is an involution. A simple computation yields the equation
 $B \bar{b} = \gamma I$. From this it follows that $B$ is a diagonal matrix with diagonal entries in $S^1$ and of determinant 1.

Evaluating now the relation (\ref{relsigma1}) one obtains with little effort  the equation $B_{33}^3 = 1$. Hence, after pulling out $B_{33}I$ from B we can assume 
 without loss of generality that $B_{33} = 1 $ holds.

Evaluating all this we see that $B$ is, without loss of generality, 
 a diagonal matrix of the form $B = (b,b^{-1},1)$ with $b \in S^1$.

But now it is straightforward to verify that 
 $D=\left( \sqrt{b}^{\frac{-1}{2}}, \sqrt{b}^{\frac{1}{2}}, 1\right)$
satisfies 
\begin{equation*}
\ad (B) \sigma \ad (B)^{-1} = \sigma 
 \quad \mbox{and} \quad \ad (B) \tau \ad (B)^{-1}  = \beta.
\end{equation*}
This proves the claim.

(b)  By evaluating the first line in Theorem 1.1 we know that $\tau_0$ commutes with $\sigma$.
Hence we obtain $\sigma \circ \ad (Q) \sigma = \ad (Q).$ But then we also obtain 
 $\sigma^2  \circ \ad (Q) \sigma^2 = \ad (Q).$ Similar to the proof of the last proposition we conclude from his that $Q$ is a generalized permutation matrix, more precisely 
 belonging to a transposition. Moreover, the equation
  $\sigma \circ \ad (Q) \sigma = \ad (Q).$ leads to $P = \nu Q P^T Q^T.$ For the underlying permutation matrices this implies  $\hat{P} = \hat{Q} \hat{ P}^T \hat{Q}^T.$
  Since $P$ and $Q$ are transpositions we conclude $\hat{P} = \hat{Q}$.
Evaluating now  $\sigma \circ \ad (Q) \sigma = \ad (Q)$ one  obtains that all entries of
$Q$ are sixth roots of unity and have the same square. Finally evaluating that $\tau$ is an involution we obtain after a simple computation  $Q_{33} = -1$ and the other two entries are 
equal and $\pm 1$. If they are equal to 1, then we have shown $Q = P_0$. If they are $-1$, then we conjugate $\tau$ and $\sigma$ by $\ad \di (-1,-1,1)$ and observe that this does not change $\sigma$ and brings $\tau$ into the form $\ad (P_0) \tau_0.$
\end{proof}

  \begin{Corollary}
 The cases $(\maltese_{\ST})$ and $(\maltese_{\SIA})$ 
 in Theorem $\ref{classify}$ are exactly all possible geometric cases, 
 where $\tau$ and $\sigma$ satisfy $\sigma \tau \sigma = \tau$.
 \end{Corollary}

\bibliographystyle{plain}
%\bibliography{../math}
\def\cprime{$'$}

\end{document}